\documentclass[11pt, reqno]{amsart}
\usepackage{amsmath, amsthm, amscd, amsfonts, amssymb, graphicx, color}
\usepackage[bookmarksnumbered, colorlinks, plainpages]{hyperref}
\numberwithin{equation}{section}
\DeclareMathOperator*{\essinf}{ess\,inf}
\DeclareMathOperator*{\esssup}{ess\,sup}

\newtheorem{theo}{Theorem}[section]
\newtheorem{defi}[theo]{Definition}

\newtheorem{prop}[theo]{Proposition}
\newtheorem{nota}[theo]{Notation}
\newtheorem{coro}[theo]{Corollary}
\newtheorem{lem}[theo]{Lemma}
\newtheorem{rem}[theo]{Remark}

\renewenvironment{proof}[1][Proof]{\textbf{#1. } }{\ \rule{0.5em}{0.5em}}

\begin{document}
\setcounter{page}{1}

\title[Normed K\"{o}the space structure and  K\"{o}the duals of Fofana spaces $\cdots$]{Normed K\"{o}the space structure and  K\"{o}the duals of variable-exponent Fofana spaces and their preduals} 

\author[N. Diarra et al]{ Nouffou Diarra$^{1*}$ . Pokou Nagacy$^2$ . B\'{e}renger Akon Kpata$^3$} 
\maketitle

\address{$^{1}$ Laboratoire de Math\'{e}matiques et Applications, UFR Math\'{e}matiques et Informatique, Universit\'{e} F\'{e}lix Houphou\"{e}t-Boigny, 22 BP 582 Abidjan 22, C\^ote d'Ivoire.}\\
e-mail : nouffoud$@$yahoo.fr  , ORCID : https://orcid.org/0000-0003-1123-506X\\

\address{ $^{2}$ Laboratoire des Sciences et Technologies de l'Environnement, UFR Environnement, Universit\'{e}
	Jean Lorougnon Gu\'{e}d\'{e}, BP 150 Daloa, C\^{o}te d'Ivoire.}\\
e-mail : pokounagacy@yahoo.com , ORCID :  https://orcid.org/0009-0004-2769-4804 \\

\address{ $^{3}$ Laboratoire de Math\'ematiques et Informatique, UFR Sciences Fondamentales et Appliqu\'ees,
               	Universit\'e Nangui Abrogoua, 02 BP 801 Abidjan 02, C\^ote d'Ivoire.}\\
e-mail : kpata\_akon@yahoo.fr  , ORCID :  https://orcid.org/0009-0005-6849-251X\\

\address{\textbf{*Correspondence:} Nouffou Diarra \textnormal{(nouffoud$@$yahoo.fr)}}\\










\begin{abstract}
 Fofana spaces, introduced on the basis of Wiener amalgam spaces, have attracted some interest recently and have found applications in the field of partial differential equations. In this paper, we study the normed K\"{o}the structure of the predual spaces $\mathcal{H}(p^{\prime}(\cdot),q^{\prime},\alpha^{\prime})$ of Fofana spaces with variable exponent $(L^{p(\cdot)}, \ell^{q})^{\alpha}$, where $1\leq p(\cdot) \leq \alpha \leq q\leq \infty$. As an application of this investigation, we identify the K\"{o}the duals and K\"{o}the biduals of these two spaces, and we obtain a predual of  $\mathcal{H}(p^{\prime}(\cdot),q^{\prime},\alpha^{\prime})$. Finally, we prove that neither  $(L^{p(\cdot)}, \ell^{q})^{\alpha}$ nor $\mathcal{H}(p^{\prime}(\cdot),q^{\prime},\alpha^{\prime})$ is reflexive.
\end{abstract} 


\section{Introduction} \label{Sec1}

 Lebesgue, Lorentz, Orlicz and various Morrey-type spaces constitute some of the most important examples of function spaces arising in modern analysis. Their flexibility in describing different forms of integrability has made them indispensable in harmonic analysis, interpolation theory, operator theory, and the analysis of partial differential equations (see, for example, \cite{Ben-Sha,Petteri-2019,Sawano-Fazio-Hakim-2020} and the references therein). A unifying framework for many of these spaces is provided by the theory of Banach function spaces and, more generally, K\"{o}the function spaces, whose lattice structure allows one to investigate not only norm and integrability properties but also order-theoretic and duality phenomena. 
Note that the theory of normed K\"{o}the function spaces has been initiated by G. K\"{o}the in the thirties and further developed by W. A. J. Luxemburg \cite{Lux-1955}, A. C. Zaanen \cite{Zaanen}, C. Bennett and R. C. Sharpley \cite{Ben-Sha}, and many others. 
 Although the usefulness of this theory is not immediately apparent in classical Lebesgue spaces case, due to the very simple structure of Lebesgue spaces, it has proved to be one of the main tools of modern functional analysis. Its abstract setting allows many results established for classical Lebesgue spaces to be extended to much broader classes of function spaces. 
 
Within the framework of the theory of normed K\"{o}the function spaces, two fundamental structural properties are the absolute continuity of the norm and the Fatou property. The interplay between these two concepts is particularly significant.  On the one hand, the  absolute continuity of the norm provides a precise description of the interaction between the norm and the measure structure of the underlying space. It guarantees that the contribution of functions supported on sets of sufficiently small measure becomes negligible in norm, which is fundamental for density, approximation and duality arguments. On the other hand, the Fatou property ensures that monotone limits of uniformly bounded positive sequences remain inside the space and that the norm behaves continuously under monotone convergence. Consequently, the Fatou property plays a decisive role in the completeness of K\"{o}the spaces and in the duality theory.
 Together, the absolute continuity of the norm and the Fatou property provide the analytical foundation required to extend several classical theorems from Lebesgue spaces to more general Banach function spaces. These properties are particularly relevant for amalgam spaces, which combine local and global integrability conditions and therefore exhibit a richer structure than classical Lebesgue spaces. Consequently, determining whether a given amalgam space possesses the structure of a K\"{o}the function space, and establishing the precise conditions under which its norm has the Fatou property or is absolutely continuous, are natural and significant questions. Such an investigation provides a systematic framework for understanding the duality and approximation properties of amalgam spaces.
 
 In this paper, we are concerned with some amalgam spaces with variable exponent as well as  their  predual spaces. We aim to study the normed K\"{o}the structure of these spaces and, subsequently, to establish some duality results concerning them. First, we shall provide a brief historical overview of the development of these spaces. 

A variable exponent on $\mathbb{R}^d$ is a measurable function $p$ from $\mathbb{R}^d$ to $[1,\infty]$. In order to distinguish between variable and constant exponents, we shall always denote variable exponent by $p(\cdot)$.
In the remainder of this paper we shall use the following notations. \\
$\bullet$ The conjuguate exponent $p^{\prime}(\cdot)$ of $p(\cdot)$ is defined by the formula $\frac{1}{p^{\prime}(\cdot)}=1-\frac{1}{p(\cdot)}$ with the convention $\frac{1}{\infty}=0$.\\
$\bullet$ $\mathcal{P}(\mathbb{R}^d)$ denotes the set of all variable exponents $p(\cdot)$ on $\mathbb{R}^d$ such that $1<p_{-} \leq p(\cdot)\leq p_{+}<\infty$, where
\begin{align*}
	p_{-}=\essinf_{x \in  \mathbb{R}^d} p(x)  \mbox{ and } p_{+}=\esssup_{x \in  \mathbb{R}^d} p(x).
\end{align*}
$\bullet$ For any measurable subset $\Omega$ of $\mathbb{R}^d$, $|\Omega|$ and $\chi_{\Omega}$ stand for its Lebesgue measure and its characteristic function, respectively.

Let $p(\cdot)$ be a variable exponent on $\mathbb{R}^d$.
The Lebesgue space with variable exponent $L^{p(\cdot)}:=L^{p(\cdot)}(\mathbb{R}^d)$ consists of all measurable functions $f$ such that for some $\lambda>0$,
\begin{align*}
	\int_{\mathbb{R}^d}\left( \frac{|f(x)|}{\lambda }\right)^{p(x)}dx<\infty.
\end{align*}
The space $L^{p(\cdot)}$ becomes a Banach space when equipped with the Luxemburg norm
\begin{align*}
	\| f\|_{p(\cdot)}:=\inf \left\{ \lambda>0 :\int_{\mathbb{R}^d}\left( \frac{|f(x)|}{\lambda }\right)^{p(x)}dx \leq 1 \right\}.
\end{align*}
If  $p(\cdot)=p$ is a constant function then $\| \cdot\|_{p(\cdot)}$ coincides with the classical  Lebesgue norm $\| \cdot\|_{p}$ and so $L^{p(\cdot)}$ is equal to the classical Lebesgue space $L^{p}$.  
The space  $L^{p(\cdot)}_{\rm loc}:=L^{p(\cdot)}_{\rm loc}(\mathbb{R}^d)$ is defined by
\begin{align*}
	L^{p(\cdot)}_{\rm loc}=\left\{ f : f\chi_{K}\in L^{p(\cdot)}, \mbox{ for all compact subsets $K$ of $\mathbb{R}^d$}\right\}.
\end{align*} 

Let $f \in L^{p(\cdot)}_{\rm loc}$, $1\leq q \leq \infty$ and $r>0$.  We set 
\begin{align}
	_{r}\|f\|_{p(\cdot),q}:=\left \| \left\{ \|f\chi_{I_{k}^{r}}\|_{p(\cdot)} \right\}_{k\in \mathbb{Z}^d}\right\|_{\ell^{q}}=
	\left\{
	\begin{array}{rl}
		&\displaystyle \left(\sum_{k\in \mathbb{Z}^d} \|f\chi_{I_{k}^{r}}\|^{q}_{p(\cdot)} \right)^{\frac{1}{q}} \;\;\;\mbox{if}\;\; q<\infty ,\\
		&\displaystyle\sup_{k\in \mathbb{Z}^d} \|f\chi_{I_{k}^{r}}\|_{p(\cdot)} \;\;\;\;\;\;\;\;\;\;\;\;\mbox{if}\;\; q =\infty,
	\end{array}
	\right.
\end{align}
where
\begin{align*}
	I_{k}^{r}=\displaystyle \prod_{j=1}^{d}\left[ rk_{j},r(k_{j}+1) \right)\;, \;\;\;\; k=(k_{1},...,k_{d})\in \mathbb{Z}^{d}.
\end{align*}
When $r=1$, we shall simply write  $\|f\|_{p(\cdot),q}$ instead of $_{1}\|f\|_{p(\cdot),q}$.
The variable exponent  amalgam space $\left(L^{p(\cdot)},\ell^{q}\right):=(L^{p(\cdot)}, \ell^{q})(\mathbb{R}^d)$ is defined by  
\begin{align*}
	(L^{p(\cdot)}, \ell^{q})= \left\{ f\in L^{p(\cdot)}_{\rm{loc}}\;:\;\|f\|_{p(\cdot),q}<\infty \right\}.
\end{align*}
If $p(\cdot)=p$ is a constant exponent then $(L^{p(\cdot)}, \ell^{q})$ coincides with the classical Wiener amalgam space
$(L^{p}, \ell^{q})$, introduced by N. Wiener \cite{Wie-1926} since 1926. Although, the amalgam space $(L^{p}, \ell^{p})$ coincides with the usual Lebesgue space $L^{p}$,
for any $(\alpha,\rho)\in (0,\infty]\times (0,\infty) $, the dilation operator $St_{\rho}^{(\alpha)}$, defined by
\begin{align}\label{eq.2.4}
	St_{\rho}^{(\alpha)}f(x)= \rho^{-\frac{d}{\alpha}}f\left(\rho^{-1} x \right), \;\; f \in  L^{1}_{\rm{loc}} \mbox{ and }x\in \mathbb{R}^{d},
\end{align}
is not isometric in proper amalgam spaces, contrary to Lebesgue spaces. In order to compensate for this shortfall, in 1988, I. Fofana \cite{Fof-1988} introduced the function spaces $(L^{p}, \ell^{q})^{\alpha} \; (1\leq p \leq \alpha \leq q\leq \infty)$, called after Fofana spaces, which consist of  functions $f\in \left(L^{p},\ell^{q}\right)$ satisfying 
\begin{align*}
\sup_{\rho >0}\; \left \|St_{\rho}^{(\alpha)}f \right\|_{p,q}<\infty.
\end{align*}
 The Fofana space $(L^{q}, \ell^{p})^{\alpha}$ can be viewed as a bridge connecting the Lebesgue space $L^{\alpha}$ with the Morrey space $\mathcal{M}_{q}^{\alpha}$.  Actually, I. Fofana \cite{Fof-1988} proved that,  for $1\leq q < \alpha \leq \infty$ fixed and $p\in [\alpha,\infty]$, the spaces $(L^{q}, \ell^{p})^{\alpha}$ form a chain of distinct Banach spaces starting  with $L^{\alpha}=(L^{q}, \ell^{\alpha})^{\alpha}$  and ending by $\mathcal{M}_{q}^{\alpha}=(L^{q}, \ell^{\infty})^{\alpha} $. We recall that the Morrey space $\mathcal{M}_{q}^{\alpha} \;(1\leq q \leq \alpha < \infty)$ is defined as the set of all  measurable functions $f$  such that
\begin{align}
	\|f\|_{\mathcal{M}_{q}^{\alpha}}:=\sup_{x\in \mathbb{R}^{d},\,r>0}  r^{d\left(\frac{1}{\alpha}-\frac{1}{q}\right)}\left(\int_{Q(x,r)} |f(y)|^{q} dy\right)^{\frac{1}{q}}
\end{align}
is finite, where $Q(x,r)$ denotes the cube centred at $x \in \mathbb{R}^{d}$ with side length $r$, whose sides are parallel to the axes of coordinates.
Many classical results in Fourier analysis, obtained in Lebesgue or Morrey spaces, have been extended in the setting of Fofana spaces. For background on these spaces we refer the reader to \cite{Coul-Fof-2019,Dia-Fof-2023,Dia-Fof-2024,Dosso-Fof-San,Feuto-BP,Feuto-JFA,Fofana 2001,Kpata-Fofana,San-Fof,Wei} and  the references therein.  

The Fofana spaces  $(L^{q}, \ell^{p})^{\alpha}$ were studied in recent years quite intensively and systematically. In 2019, H. G. Feichtinger and J. Feuto \cite{Fei-Feu} introduced  their predual spaces $\mathcal{H}(p^{\prime},q^{\prime},\alpha^{\prime})$, where for  $1\leq s \leq \infty$, $s'$ denotes  the conjugate exponent of $s$ : $\;\frac{1}{s'}=1-\frac{1}{s}$. Later, in 2021, N. Diarra and I. Fofana \cite{Dia-Fof-2021} investigated the normed K\"{o}the structure of  $\mathcal{H}(p^{\prime},q^{\prime},\alpha^{\prime})$ and obtained duality results, and the boundedness of some classical operators on both $(L^{p}, \ell^{q})^{\alpha}$ and $\mathcal{H}(p^{\prime},q^{\prime},\alpha^{\prime})$.
Furthermore, Fofana spaces have been generalized in several directions, including the incorporation of a variable exponent $p(\cdot)$ in place of the constant exponent $p$. Their predual spaces are also known. Let us recall their definitions.

\begin{defi}
	Let $p(\cdot)$ be a variable exponent on $\mathbb{R}^d$ and $ 1 \leq \alpha , q \leq \infty$. 
	The variable exponent  Fofana space $\left(L^{p(\cdot)},\ell^{q}\right)^{\alpha}:=\left(L^{p(\cdot)}, \ell^{q}\right)^{\alpha}(\mathbb{R}^d)$ is defined by 
	\begin{align*}
		\left(L^{p(\cdot)}, \ell^{q}\right)^{\alpha}= \left\{ f\in L^{p(\cdot)}_{\rm{loc}}\;:\;\|f\|_{p(\cdot),q,\alpha}<\infty \right\},
	\end{align*} 
	where
	\begin{align}
		||f||_{p(\cdot),q,\alpha}=\sup_{\rho >0}\; \left \|St_{\rho}^{(\alpha)}f \right\|_{p(\cdot),q}.
	\end{align}
\end{defi}

\begin{defi}
	Let $p(\cdot)$ be a variable exponent on $\mathbb{R}^d$ and $ 1 \leq q \leq \alpha \leq p(\cdot)$.
	\begin{enumerate} 
		\item  A sequence $\left\{(c_{n},\rho_{n},f_{n})\right\}_{n\geq 1}$ of elements of $\;\mathbb{C}\times (0, \infty)\times (L^{p(\cdot)}, \ell^{q})$ is called a $\mathfrak{h}$-decomposition of a function $f$  if $f=\displaystyle\sum_{n\geq 1} c_{n}St_{\rho_{n}}^{(\alpha)}f_{n}$ in the sense of $L^{p(\cdot)}_{\rm{loc}}$ and
		\begin{align*}
			\left\{
			\begin{array}{rl}
				&\|f_{n}\|_{p(\cdot),q}\leq 1\;,\;\;\;\;\;n\geq 1,\\\\
				&\displaystyle \sum_{n\geq 1}|c_{n}|< \infty.
			\end{array}
			\right.
		\end{align*}
		\item The space $\mathcal{H}(p(\cdot),q,\alpha):=\mathcal{H}(p(\cdot),q,\alpha)(\mathbb{R}^d)$ is defined as the set of all $f \in L^{p(\cdot)}_{\rm{loc}}$  whose set of  $\mathfrak{h}$-decompositions is non-empty.
		\item The space $\mathcal{H}(p(\cdot),q,\alpha)$ can be endowed with the norm $\widetilde{\|\cdot \|}_{\mathcal{H}(p(\cdot),q,\alpha)}$ defined, for all $f\in \mathcal{H}(p(\cdot),q,\alpha)$, by
		\begin{align*}
			\widetilde{\|f \|}_{\mathcal{H}(p(\cdot),q,\alpha)}=\inf \left\{\displaystyle \sum_{n\geq 1}|c_{n}| : f=\displaystyle\sum_{n\geq 1} c_{n}St_{\rho_{n}}^{(\alpha')}f_{n}\right\}.
		\end{align*}
\end{enumerate}
\end{defi} 

In the setting of both variable exponent Fofana spaces and their preduals, the reader is referred to \cite{Dia-2023,yang-zhou-23} for basic properties and the boundedness properties of classical operators, including the Hardy-Littlewood maximal operator, fractional operators, and Calder\'{o}n-Zygmund operators.
However, to the best of our knowledge, these spaces have yet to be 
investigated from the viewpoint of K\"{o}the function space theory. Thereby, an important issue is to understand their normed  K\"{o}the structure. 

The objective of this paper is to present a systematic study of the normed K\"{o}the space structure of $\mathcal{H}(p(\cdot),q,\alpha)$, as well as variable exponent Fofana spaces. Mainly, we investigate  the absolute continuity of the norm, Fatou property, separability, K\"{o}the dual,  K\"{o}the bidual  and reflexivity in the framework of these spaces.


The paper is organized as follows. In Section \ref{Sec2}, we recall the necessary background on normed K\"{o}the  and  variable exponent spaces. Section \ref{Sec3} is devoted to the Fatou property and absolute continuity of the norm in Fofana spaces with variable exponent. In Section \ref{Sec4} we investigate the normed K\"{o}the structure of $\mathcal{H}(p(\cdot),q,\alpha)$, including the absolute continuity of the norm, Fatou property and separability. 
Section \ref{Sec5} deals with  duality and reflexivity questions.


\section{Preliminaries} \label{Sec2}

\subsection{Background on normed K\"{o}the spaces}  \label{Subsec2.1}
 By $L^{0}:=L^{0}(\mathbb{R}^d)$, we denote the set of equivalence classes (modulo equality almost everywhere) of  measurable complex-valued functions on $\mathbb{R}^d$. $L^{0}_{+}$ stands for the set of all nonnegative elements of $L^{0}$.
Here, we use the terminology from the theory of Banach function spaces introduced by Zaanen in \cite[Chapter 15]{Zaanen}.

\begin{defi} \label{defi 2.1} (Function norm)
	\begin{enumerate}
\item A function norm on $\mathbb{R}^d$ is a map $\sigma$ of  $L^{0}_{+}$ into $[0, \infty]$ such  that, for any  elements $f$ and $g$ of $L^{0}_{+}$ and  any real number $\lambda \geq 0$, we have:
\begin{enumerate}
\item	$\sigma(f)=0\;\Longleftrightarrow \; f=0$  in $L^{0}$ (strict positivity);
\item   $\sigma(\lambda f)=\lambda\;\sigma(f)$ (homogeneity);
\item	$\sigma(f+g)\leq \sigma(f)+\sigma(g)$ (triangle inequality);
\item   $f\leq g \mbox{ in }L^{0}\Longrightarrow \sigma(f)\leq \sigma(g)$ (lattice property).
\end{enumerate}
\item	A function norm $\sigma$ on $\mathbb{R}^d$ is said to be saturated if every measurable subset $E$ of $\mathbb{R}^{d}$, with positive measure $|E|$, contains a measurable subset $F$ such that $0<|F|$ and $\sigma(\chi_{F})<\infty$.
\item A function norm $\sigma$ on $\mathbb{R}^d$ has the Riesz-Fischer property whenever, for any sequence $\left( f_{n}\right)_{n\geq 1}$ of elements of $L^{0}_{+}$, if  $\;\displaystyle \sum_{n\geq 1}\sigma(f_{n})<\infty$ then  $\sigma \left( \displaystyle \sum_{n\geq 1}f_{n}\right) <\infty$.	
\end{enumerate}
\end{defi}

\begin{defi} (Normed K\"{o}the space).
	Let $\sigma$ be a function norm on $\mathbb{R}^d$.
	\begin{enumerate}
	\item  The set $X^{\sigma}=\left\{ f\in L^{0}:\sigma(|f|)<\infty \right\}$ is called the normed K\"{o}the space on $\mathbb{R}^d$ defined by $\sigma$.
	\item A Banach function space on $\mathbb{R}^d$  is  a  normed K\"{o}the space $X^{\sigma}$ on $\mathbb{R}^d$ such that $(X^{\sigma},\|\cdot\|_{X^{\sigma}})$ is a Banach space, where the norm $\|\cdot\|_{X^{\sigma}}$ on $X^{\sigma}$ is defined by
	\begin{align*}
		\|f\|_{X^{\sigma}}=\sigma(|f|)\;,\;\;\;\;\;f\in X^{\sigma}.
	\end{align*}
\end{enumerate}
\end{defi}

\begin{defi} \label{defi 2.2}
	Let $\sigma$ be a function norm on $\mathbb{R}^d$. The space $X$ is the corresponding normed K\"{o}the space  and $\|\cdot\|_{X}=\sigma(|\cdot|)$.
	\begin{enumerate}
	\item $X$ satisfies the Fatou property if  for any sequence $(f_{n})_{n\geq 1}$ of elements of $X$, 
		\begin{align*}
			\left[0\leq f_{n} \uparrow f \mbox{ in }\;L^{0}\;\mbox{ and } \displaystyle \sup_{n\geq 1}\|f_{n}\|_{X}<\infty \right]\Longrightarrow \left[ f\in X\;\mbox{ and } \displaystyle \lim_{n\rightarrow \infty} \|f_{n}\|_{X}=\|f\|_{X}\right].
		\end{align*}
	\item Let us assume that $\sigma$ is saturated.
		\begin{enumerate}
	\item  A function $f$ of $X$ is of absolutely continuous norm if $\displaystyle \lim_{n\rightarrow \infty}\|f_{n}\|_{X}=0$ for every sequence $(f_{n})_{n\geq 1}$ in $X$  such that 
	\begin{align*}
	\left\{
	\begin{array}{rl}
		&|f|\geq f_{n}\geq f_{n+1}\;\;\mbox{ in }\;L^{0}\;,\;\;\;\;n\geq 1 \\\\
		&\displaystyle\lim_{n\rightarrow\infty}f_{n}=0\;\mbox{ in }\;L^{0}.
	\end{array}
	\right.
\end{align*}
   \item  $\sigma$ is said to be  an absolutely continuous norm (or $X$ has an absolutely continuous norm) if every element of $X$ is of absolutely continuous norm.
\end{enumerate}
\end{enumerate}
\end{defi}

\begin{defi} (K\"{o}the dual space). 	Let $X$ be a normed K\"{o}the space on $\mathbb{R}^d$. 
	\begin{enumerate}
	\item The K\"{o}the dual space (or associate space) $X'$ of $X$ is  defined to be the collection of all elements $g$ of $L^{0}$ such that
\begin{align}
	\|g\|_{X'} :=\sup\left\{ \int_{\mathbb{R}^{d}} \left|g(x)f(x)\right|dx\;:\;f \in X \mbox{ and } \|f\|_{X}\leq 1 \right\}<\infty.
\end{align}
\item  The K\"{o}the bidual space (or second associate space) $X''$ of $X$ is  defined to be the collection of all elements $g$ of $L^{0}$ such that
\begin{align}
	\|g\|_{X''} :=\sup\left\{ \int_{\mathbb{R}^{d}} \left|g(x)f(x)\right|dx \;:\;f \in X' \mbox{ and } \|f\|_{X'}\leq 1 \right\}<\infty.
\end{align}
	\end{enumerate}
\end{defi}

\begin{rem}
	Let $X$ be a normed K\"{o}the space on $\mathbb{R}^d$, $g\in X'$ and for any $f\in X$
	 \begin{align}
		T_{g}(f)=\int_{\mathbb{R}^{d}}f(x)g(x)dx.
	\end{align}
	It is well-known that $g\mapsto T_{g}$ is an isometric linear map of $X'$ into  the topological dual space $X^{*}$ of $X$  (see \cite[Theorem 2 \S \; 69 ]{Zaanen}).
	Therefore we shall look $X'$ as a closed subspace of $X^{*}$ by identifying any element $g$ of $X'$ with $T_{g}$.
\end{rem}

Let us recall the following results. 

\begin{prop} \cite[Chapter 15]{Zaanen}. \label{p2.6}
		Let $\sigma$ be a function norm on $\mathbb{R}^d$ and $X$ be the corresponding normed K\"{o}the space  such that $\|\cdot\|_{X}=\sigma(|\cdot|)$.
\begin{enumerate}
	\item $X$ is a Banach space if and only if $\sigma$ has the Riesz-Fischer property.
	\item The K\"{o}the dual space $X'$ of $X$, equipped with $\|\cdot\|_{X'}$, is a Banach function space satisfying the Fatou property.
	\item The K\"{o}the bidual space $X''$ of $X$ coincides with $X$ if and only if $X$ has the Fatou property.
	\item Assume that $\sigma$ is saturated and $f$ is an element of $ X$. Denote by $X_{\rm a}$ the set of all elements of $X$ of absolutely continuous norm.
	\begin{enumerate}
	\item $f$ belongs to $ X_{\rm a}$ if  and only if, for every  sequence  $(f_{n})_{n\geq 1}$  of elements of $X$ and any measurable function $g$ such that, on $\mathbb{R}^{d}$, $|f_{n}|\leq |f|\;(n\geq 1)$ and $\displaystyle \lim_{n\rightarrow \infty} f_{n}= g$,
	then we have  $\displaystyle \lim_{n\rightarrow \infty}\|f_{n}-g\|_{X}=0.$
    \item $f$ belongs to $ X_{\rm a}$ if  and only if
	$\displaystyle \lim_{n\rightarrow \infty}\|f\chi_{E_{n}}\|_{X}=0\;$ for any non-increasing 
	sequence $(E_{n})_{n\geq 1}$  of subsets of $\mathbb{R}^{d}$ such that $ \left|\displaystyle \bigcap_{n\geq 1}E_{n}\right|=0$.
		\item  The topological dual space $X^{*}$ and the K\"{o}the dual space $X'$ of $X$ coincide  if and only if $X=X_{\rm a}$, that is  $\sigma$ is an absolutely continuous norm.
		\item The topological dual space $X_{\rm a}^{*}$ of $X_{\rm a}$ can be identified with the K\"{o}the dual space $X'$ of $X$   if and only if the carrier of $X_{\rm a}$ is the whole set $\mathbb{R}^{d}$, where the carrier of $X_{\rm a}$ is defined to be  $\mathbb{R}^{d}\backslash {F}$ with $F$ being the maximal subset of $\mathbb{R}^{d}$ on which every element of $X_{\rm a}$ vanishes almost everywhere.
		\item $X$ is reflexive if and only if both $X$ and $X'$ have absolutely continuous norms and satisfy the Fatou property.
	\end{enumerate}
\end{enumerate}	
\end{prop}


\subsection{Background on variable exponent spaces}  \label{Subsec2.2}

Variable exponent Lebesgue spaces have many properties in common with the  classical Lebesgue spaces, but they also differ in some ways. For instance, we have the following results.

\begin{prop} \cite{Cruz-Fior-2013} \label{p2.1}
	Let $p(\cdot), q(\cdot) \in \mathcal{P}(\mathbb{R}^d)$ and $\Omega$ a subset of $\mathbb{R}^d$ such that $|\Omega|<\infty$. 
	\begin{enumerate}
		\item  If $f \in L^{p(\cdot)} $ and $g \in L^{p^{\prime}(\cdot)} $ then $fg \in L^{1}$ and 
	\begin{align} \label{eq2.4}
			\int_{\mathbb{R}^{d}}|f(x)g(x)|\,dx \leq  \left( 1+1/p_{-}-1/p_{+}\right) \|f\|_{p(\cdot)}\; \|g\|_{p^{\prime}(\cdot)}.
		\end{align}
		\item  If $p(\cdot) \leq q(\cdot)$  then $L^{q(\cdot)}(\Omega) \subseteq L^{p(\cdot)}(\Omega)$ and 
	\begin{align} \label{eq2.5'}
			\|f\|_{p(\cdot)} \leq \left(1+|\Omega| \right) \, \|f\|_{q(\cdot)}, \quad f\in L^{q(\cdot)}(\Omega).
		\end{align}
		\item  There exist constants $C_{1}, C_{2}>0$ such that
		\begin{align} \label{eq2.5}
	 C_{1} \, \|f\chi_{\Omega}\|_{p_{-}}	\leq \|f\chi_{\Omega}\|_{p(\cdot)} \leq C_{2}\, \|f\chi_{\Omega}\|_{p_{+}}, \quad f\in L^{p(\cdot)}_{\rm loc}.
		\end{align}
		\item We have
		 \begin{align} 	\label{eq2.6}
		 		\|\chi_{\Omega}\|_{p(\cdot)} \leq \max \left( |\Omega|^{\frac{1}{p_{-}}}, |\Omega|^{\frac{1}{p_{+}}} \right).	
		\end{align}
	\end{enumerate}
\end{prop}

The following results are well known for variable exponent  amalgam spaces.
\begin{prop} \cite{Aydin-Unal-2019,VEAS,yang-zhou-23}\label{P2.1}.
	Let $p(\cdot)$ be a variable exponent on $\mathbb{R}^d$ and $1\leq q \leq \infty$.
	\begin{enumerate}
		\item $\left(L^{p(\cdot)},\ell^{q}\right)$, endowed with $\|.\|_{p(\cdot),q}$, is a  Banach linear subspace of $L^{p(\cdot)}_{\rm{loc}}$.
		\item If $f \in \left(L^{p(\cdot)},\ell^{q}\right) $ and $g \in \left(L^{p^{\prime}(\cdot)},\ell^{q^{\prime}}\right) $ then $fg \in L^{1}$ and there exists $C>0$ such that
		\begin{align}
			\int_{\mathbb{R}^{d}}|f(x)g(x)|\,dx \leq C \|f\|_{p(\cdot),q}\; \|g\|_{p'(\cdot),q'}.
		\end{align}
		\item If $p(\cdot) \in \mathcal{P}(\mathbb{R}^d)$ and $1\leq q<\infty$ then the set $\mathcal{C}_{\rm{c}}:=\mathcal{C}_{\rm{c}}(\mathbb{R}^d)$ of all continuous and compactly supported functions on $\mathbb{R}^{d}$ is a dense subspace of $\left(L^{p(\cdot)},\ell^{q}\right)$.  
	\end{enumerate}
\end{prop}

Let us recall some basic properties of Fofana spaces with variable exponent.

\begin{prop}  \label{p2.9}\cite{Dia-2023,yang-zhou-23}.
	Let $p(\cdot) , p_{1}(\cdot)\in \mathcal{P}(\mathbb{R}^d)$ and $ 1 \leq \alpha , q, q_{1} \leq \infty$.
	\begin{enumerate}
		\item $\left(L^{p(\cdot)},\ell^{q}\right)^{\alpha}$ is non trivial if and only if $ p(\cdot) \leq \alpha \leq q$.
		\item $\left(L^{p(\cdot)},\ell^{q}\right)^{\alpha}$, endowed with $\|.\|_{p(\cdot),q,\alpha}$, is a Banach space  continuously embedded in $\left(L^{p(\cdot)},\ell^{q}\right)$.
    	\item If $p(\cdot)=p$ then $\left(L^{p(\cdot)},\ell^{q}\right)^{\alpha}$ is just the classical Fofana space $\left(L^{p},\ell^{q}\right)^{\alpha}$.
		\item If $ p(\cdot) \leq \alpha \leq q \leq \infty$ then 
		$\left(L^{p(\cdot)},\ell^{q}\right)^{\alpha}$ is a solid space; that is : if $g\in L^{0}$ and  $f\in \left(L^{p(\cdot)},\ell^{q}\right)^{\alpha}$ such that $|g| \leq |f|$  a.e. then $\|g\|_{p(\cdot),q,\alpha}\leq \|f\|_{p(\cdot),q,\alpha}$ and therefore $g$ belongs to $\left(L^{p(\cdot)},\ell^{q}\right)^{\alpha}$. 
		\item If $p(\cdot) \leq p_{1}(\cdot) \leq \alpha \leq q \leq \infty$ then
		\begin{align} \label{eq2.8}
			\|f\|_{p(\cdot),q,\alpha} \leq   2 \, \|f\|_{p_{1}(\cdot),q,\alpha},\;\;\;\;\;\;\; f \in L^{p(\cdot)}_{\rm{loc}}
		\end{align}
		and therefore $\left(L^{p_{1}(\cdot)},\ell^{q}\right)^{\alpha}$ is continuously embedded in $\left(L^{p(\cdot)},\ell^{q}\right)^{\alpha}$.
		\item If $p(\cdot) \leq \alpha \leq q \leq q_{1} \leq \infty$ then
		\begin{align} \label{eq2.9}
			\|f\|_{p(\cdot),q_{1},\alpha} \leq  \|f\|_{p(\cdot),q,\alpha},\;\;\;\;\;\;\; f \in L^{p(\cdot)}_{\rm{loc}}
		\end{align}
		and therefore $\left(L^{p(\cdot)},\ell^{q}\right)^{\alpha}$ is continuously embedded in $\left(L^{p(\cdot)},\ell^{q_{1}}\right)^{\alpha}$.
	\end{enumerate}
\end{prop}

We  have  the result below.

\begin{prop} \cite{yang-zhou-23}. \label{p2.11}
	Let $p(\cdot) \in \mathcal{P}(\mathbb{R}^d)$ and $ p(\cdot) \leq \alpha \leq q \leq \infty$.
	Then a predual  of  $\left(L^{p(\cdot)},\ell^{q}\right)^{\alpha}$ is the space $\mathcal{H}(p^{\prime}(\cdot),q^{\prime},\alpha^{\prime})$ in the following sense: \\
	The operator $T: g \mapsto T_{g}$, defined by 
	\begin{align} \label{eq2.11}
		T_{g}(f)= \int_{\mathbb{R}^{d}}f(x)g(x)\,dx \;,\;g \in \left(L^{p(\cdot)},\ell^{q}\right)^{\alpha} \mbox{ and } f \in \mathcal{H}(p^{\prime}(\cdot),q^{\prime},\alpha^{\prime}),
	\end{align}
	is an isometric isomorphism of $\left(L^{p(\cdot)},\ell^{q}\right)^{\alpha}$ into  $\mathcal{H}(p^{\prime}(\cdot),q^{\prime},\alpha^{\prime})^{*}$.
\end{prop}

We end this subsection by the following properties of $\mathcal{H}(p(\cdot),q,\alpha)$.

\begin{prop} \cite{Dia-2023}. \label{p2.12}
	Let $p(\cdot) \in \mathcal{P}(\mathbb{R}^d)$ and $ 1 \leq q \leq \alpha \leq p(\cdot)$.
	\begin{enumerate}
		\item	$ \left(L^{p(\cdot)},\ell^{q}\right)$  and the set $\mathcal{C}_{\rm{c}}^{\infty}:=\mathcal{C}_{\rm{c}}^{\infty}(\mathbb{R}^d)$ of all infinitely differentiable and compactly supported functions on $\mathbb{R}^{d}$  are dense subspaces of $\mathcal{H}(p(\cdot),q,\alpha)$.
		\item  $\mathcal{H}(p(\cdot),q,\alpha)$ is a solid space; that is : if $g\in L^{0}$ and  $f\in \mathcal{H}(p(\cdot),q,\alpha)$ such that $|g| \leq |f|$  a.e. then $\|g\|_{\mathcal{H}(p(\cdot),q,\alpha)}\leq \|f\|_{\mathcal{H}(p(\cdot),q,\alpha)}$ and therefore $g$ belongs to $\mathcal{H}(p(\cdot),q,\alpha)$.
	\end{enumerate}
\end{prop}


\section{The Fofana space  $\left(L^{p(\cdot)},\ell^{q}\right)^{\alpha}$} \label{Sec3}

This section is devoted to prove that the Fofana spaces  with variable exponent $\left(L^{p(\cdot)},\ell^{q}\right)^{\alpha}$  have the Fatou property and characterize their subset $\left(L^{p(\cdot)},\ell^{q}\right)^{\alpha}_{\rm a}$ of all elements of absolutely continuous norm.

\begin{theo} \label{th3.1}
Let $p(\cdot)$ be a variable exponent on $\mathbb{R}^d$ and $ p(\cdot) \leq  \alpha \leq q \leq \infty$. 
	Then  $\left(L^{p(\cdot)},\ell^{q}\right)^{\alpha}$ is a Banach function space on $\mathbb{R}^d$ satisfying the Fatou
	property.
\end{theo}

\begin{proof}
 Points (2) and (4) of Proposition \ref{p2.9} assert that  $\left(L^{p(\cdot)},\ell^{q}\right)^{\alpha}$ is a Banach space and it possesses the lattice property. Thus $\left(L^{p(\cdot)},\ell^{q}\right)^{\alpha}$ is a Banach function space on $\mathbb{R}^d$. It remains to prove that it satisfies the Fatou property.
 
 Let $(f_{n})_{n\geq 1}$ be a sequence of nonnegative elements of  $\left(L^{p(\cdot)},\ell^{q}\right)^{\alpha}$ such that $(f_{n})_{n\geq 1}\uparrow f$ and $\displaystyle \sup_{n\geq 1}\|f_{n}\|_{p(\cdot),q,\alpha}< \infty$.
	Then, $\left(\|f_{n}\|_{p(\cdot),q,\alpha}\right)_{n\geq 1}$ is a non-decreasing sequence and satisfies
	\begin{align}\label{eq3.1}
		\displaystyle \lim_{n\rightarrow \infty} \|f_{n}\|_{p(\cdot),q,\alpha}=\displaystyle \sup_{n\geq 1} \|f_{n}\|_{p(\cdot),q,\alpha}\leq \|f\|_{p(\cdot),q,\alpha}.
	\end{align}
	Let us consider a real number $t$ satisfying $0\leq t< \|f\|_{p(\cdot),q,\alpha}$. Then, there exist a positive real number $\rho$ such that $\left \|St_{\rho}^{(\alpha)}f \right\|_{p(\cdot),q}>t$ and a finite subset  $\mathbb{I}$ of $\mathbb{Z}^{d}$ for which
	\[
	\left\{
	\begin{array}{rl}
		& \displaystyle \left[\sum_{k\in \mathbb{I}} \left( \left\|\left(St_{\rho}^{(\alpha)}f\right)\chi_{I_{k}^{1}}  \right\|_{p(\cdot)}\right)^{q}  \right]^{\frac{1}{q}}>t \quad  \mbox{if } q<\infty,\\\\
		&\displaystyle \sup_{k \in \mathbb{I}}\left\|\left(St_{\rho}^{(\alpha)}f\right)\chi_{I_{k}^{1}}  \right\|_{p(\cdot)}>t \quad \quad \quad \quad \, \mbox{ if } q =\infty.
	\end{array}
	\right.
	\]
	Let $\gamma$ be the cardinality of $\mathbb{I}$ and set
		\[
		\theta =
	\left\{
	\begin{array}{rl}
		& \displaystyle \left[\sum_{k\in \mathbb{I}} \left( \left\|\left(St_{\rho}^{(\alpha)}f\right)\chi_{I_{k}^{1}}  \right\|_{p(\cdot)}\right)^{q} - t^{q} \right]\frac{1}{\gamma} \quad \mbox{if } q<\infty,\\\\
		&\displaystyle \sup_{k \in \mathbb{I}}\left\|\left(St_{\rho}^{(\alpha)}f\right)\chi_{I_{k}^{1}}  \right\|_{p(\cdot)}-t \quad \quad \quad \quad \quad \; \mbox{if } q =\infty.
	\end{array}
	\right.
	\]
	Since $(f_{n})_{n\geq 1}\uparrow f$, by the monotone convergence theorem, we get
\begin{align*}
	\left( \left\|\left(St_{\rho}^{(\alpha)}f_{n}\right)\chi_{I_{k}^{1}} \right\|_{p(\cdot)}\right)_{n \geq 1}\uparrow  \left\|\left(St_{\rho}^{(\alpha)}f\right)\chi_{I_{k}^{1}} \right\|_{p(\cdot)},
	\;\;\;\;\;\;\;\;\;k\in \mathbb{I}.
\end{align*}
Hence, there exists an integer  $n_{\theta}>0$  such that, for any integer $n\geq n_{\theta}$ and any $k\in \mathbb{I}$,
	\[
	\left\{
	\begin{array}{rl}
		&\left(\left\|\left(St_{\rho}^{(\alpha)}f_{n}\right)\chi_{I_{k}^{1}} \right\|_{p(\cdot)}\right)^{q}>\left(\left\|\left(St_{\rho}^{(\alpha)}f\right)\chi_{I_{k}^{1}} \right\|_{p(\cdot)}\right)^{q}-\theta \quad \; \mbox{if } q<\infty,\\\\
		&\left\|\left(St_{\rho}^{(\alpha)}f_{n}\right)\chi_{I_{k}^{1}} \right\|_{p(\cdot)} > \left\|\left(St_{\rho}^{(\alpha)}f \right)\chi_{I_{k}^{1}} \right\|_{p(\cdot)}-\theta \quad \quad \quad \quad \quad \mbox{if } q=\infty.
	\end{array}
	\right.
	\]
	Therefore, for any integer $n\geq n_{\theta}$, we obtain
		\[
	\left\{
	\begin{array}{rl}
		&\displaystyle \sum_{k\in \mathbb{I}} \left(\left\|\left(St_{\rho}^{(\alpha)}f_{n}\right)\chi_{I_{k}^{1}} \right\|_{p(\cdot)}\right)^{q}> \displaystyle \sum_{k\in \mathbb{I}}  \left(\left\|\left(St_{\rho}^{(\alpha)}f\right)\chi_{I_{k}^{1}} \right\|_{p(\cdot)}\right)^{q}-\gamma \,\theta=t^{q} \quad \, \mbox{if } q<\infty,\\\\
		&\displaystyle \sup_{k\in \mathbb{I}} \left\|\left(St_{\rho}^{(\alpha)}f_{n}\right)\chi_{I_{k}^{1}} \right\|_{p(\cdot)} >\displaystyle \sup_{k\in \mathbb{I}}  \left\|\left(St_{\rho}^{(\alpha)}f \right)\chi_{I_{k}^{1}} \right\|_{p(\cdot)}-\theta =t \quad \quad \quad \quad \quad \quad \mbox{if } q=\infty.
	\end{array}
	\right.
	\]
	Thus 
	\begin{align*}
		\left \|St_{\rho}^{(\alpha)}f_{n} \right\|_{p(\cdot),q} > t,\;\;\;\;\;\;\;\;\;\;n\geq  n_{\theta}.
	\end{align*}
	This shows that 
	\begin{align*}
		\|f_{n}\|_{p(\cdot),q,\alpha}=\sup_{\rho >0}\; \left \|St_{\rho}^{(\alpha)}f_{n} \right\|_{p(\cdot),q}>t, \;\;\;\;\;n\geq  n_{\theta}
	\end{align*}
	and therefore 
	\begin{align*}
		\displaystyle \lim_{n\rightarrow \infty} \|f_{n}\|_{p(\cdot),q,\alpha}> t.
	\end{align*}
	Since this holds for all $0\leq t< \|f\|_{p(\cdot),q,\alpha}$, we get
	\begin{align}\label{eq3.2}
		\displaystyle \lim_{n\rightarrow \infty} \|f_{n}\|_{p(\cdot),q,\alpha} \geq \|f\|_{p(\cdot),q,\alpha}.
	\end{align}
	From \eqref{eq3.1} and \eqref{eq3.2} we deduce that  
	\begin{align*}
		\|f\|_{p(\cdot),q,\alpha}= \displaystyle \lim_{n\rightarrow \infty} \|f_{n}\|_{p(\cdot),q,\alpha}=\displaystyle \sup_{n\rightarrow \infty} \|f_{n}\|_{p(\cdot),q,\alpha}<\infty.
	\end{align*}
	This shows that the space $\left(L^{p(\cdot)},\ell^{q}\right)^{\alpha}$ satisfies the Fatou property.
\end{proof}

The following result remains valid, as in the classical case.

\begin{prop} \label{p3.2}
	Let $p(\cdot) \in \mathcal{P}(\mathbb{R}^d)$ and $ p(\cdot) \leq  \alpha \leq q \leq \infty$. Then $L^{\alpha}$ is continuously embedded in $\left(L^{p(\cdot)},\ell^{q}\right)^{\alpha}$.
\end{prop}

\begin{proof}
	Let $f \in L^{p(\cdot)}_{\rm{loc}}$. By taking $p_{1}(\cdot)=\alpha$ in \eqref{eq2.8}, we obtain
	\begin{align*} 
		\|f\|_{p(\cdot),q,\alpha} \leq   2 \, \|f\|_{\alpha,q,\alpha}= 2 \, \|f\|_{\alpha}
	\end{align*}
	and therefore the desired result holds.
\end{proof}

The result below follows from Theorem \ref{th3.1}, Proposition \ref{p3.2} and \cite[Proposition 3.5]{Dia-Fof-2024}.
\begin{theo} \label{th3.3}
	Let $p(\cdot) \in \mathcal{P}(\mathbb{R}^d)$ and $ p(\cdot) \leq  \alpha \leq q \leq \infty$ such that $\alpha<\infty$.
	 Then the set $\left(L^{p(\cdot)},\ell^{q}\right)^{\alpha}_{\rm a}$ of all elements of $\left(L^{p(\cdot)},\ell^{q}\right)^{\alpha}$ of absolutely continuous norm is the closure of $\mathcal{C}_{\rm{c}}^{\infty}$ (and therefore of $L^{\alpha}$) in $\left(L^{p(\cdot)},\ell^{q}\right)^{\alpha}$.
\end{theo}

\begin{proof}
	From Theorem \ref{th3.1} and Proposition \ref{p3.2}, $\left(L^{p(\cdot)},\ell^{q}\right)^{\alpha}$ is a normed  K\"{o}the space in which  the Lebesgue space $L^{\alpha}$, with $1\leq \alpha<\infty$, is continuously embedded. Thus an application of \cite[Proposition 3.5]{Dia-Fof-2024} yields the desired result.
\end{proof}

\begin{rem}\label{R3.4}
	In the constant exponent case, it is shown in \cite[Remark 4.5]{Dia-Fof-2021}  that $\left(L^{p},\ell^{q}\right)^{\alpha}_{\rm a}$ differ from $\left(L^{p},\ell^{q}\right)^{\alpha}$. This fact remains true in variable exponent case; that is $\left(L^{p(\cdot)},\ell^{q}\right)^{\alpha}_{\rm a}$  is a proper subspace of $\left(L^{p(\cdot)},\ell^{q}\right)^{\alpha}$.
\end{rem}


\section{The normed K\"{o}the structure of $\mathcal{H}(p(\cdot),q,\alpha)$} \label{Sec4}

In this section, unless otherwise specified, we assume that $p(\cdot)$ is a variable exponent on $\mathbb{R}^d$ and $ 1 \leq q \leq \alpha \leq p(\cdot)\leq \infty$.

For any element $f$ of $L^{p(\cdot)}_{\rm{loc}}$, we denote by $\mathfrak{h}$-$\mathcal{D}(f,p(\cdot),q,\alpha)$  the set of all $\mathfrak{h}$-decompositions of $f$ and we set
\begin{align*}
\sigma_{p(\cdot),q,\alpha}(f)=
\left\{
\begin{array}{rl}
	&\inf \left\{\displaystyle \sum_{n\geq 1}|c_{n}| : \left\{\left(c_{n},\rho_{n},f_{n}\right)\right\}_{n\geq 1} \in \mathfrak{h}\mbox{-}\mathcal{D}(f,p(\cdot),q,\alpha)\right\} \\
	&\mbox{ if $\mathfrak{h}$-$\mathcal{D}(f,p(\cdot),q,\alpha)$}\neq \varnothing,\\\\
	&\infty \;\;\;\mbox{ if }\;\;\;\mbox{$\mathfrak{h}$-$\mathcal{D}(f,p(\cdot),q,\alpha)$}=\varnothing.
\end{array}
\right.	
\end{align*}
It is worth emphasizing that, $\sigma_{p(\cdot),q,\alpha}$ is a function norm on $\mathbb{R}^d$. Indeed,
the strict positivity, homogeneity and triangle inequality follow readily from its definition. For the lattice property one can see Point (2) of Proposition \ref{p2.12}.


\subsection{Preliminary properties}

This subsection deals with some preliminary properties of the spaces $\mathcal{H}(p(\cdot),q,\alpha)$.
\begin{prop} \label{P4.1}
	$\mathcal{H}(p(\cdot),q,\alpha)=\left\{f \in L^{p(\cdot)}_{\rm{loc}} : \sigma_{p(\cdot),q,\alpha}(|f|) <\infty \right\}$ and, for any $f \in L^{p(\cdot)}_{\rm{loc}}$, we have
\begin{align*}
\sigma_{p(\cdot),q,\alpha}(f)=	\sigma_{p(\cdot),q,\alpha}(|f|).
\end{align*}
\end{prop}

\begin{proof}
		Let $f$ be an element of $L^{p(\cdot)}_{\rm{loc}}$.\\
	$\bullet$ Suppose that $\sigma_{p(\cdot),q,\alpha}(|f|) <\infty$ and 
	$\left\{ \left(a_{n},\rho_{n},\varphi_{n}\right)\right\}_{n\geq 1}$ is any element of $\mathfrak{h}$-$\mathcal{D}(|f|,p(\cdot),q,\alpha)$.
	We set, for any integer $n\geq 1,$ $\phi_{n}=\varphi_{n}\,\mbox{sgn} f$, where, for all $x\in \mathbb{R}^{d}$,
	\[
\mbox{sgn}\,f(x)=
	\left\{
	\begin{array}{rl}
		&\frac{f(x)}{|f(x)|} \quad \; \mbox{if }f(x)\neq 0,\\\\
		&0 \quad \quad \quad \mbox{otherwise}.
	\end{array} 
	\right.
	\]	 
Clearly, $\left\{ \left(a_{n},\rho_{n},\phi_{n}\right)\right\}_{n\geq 1}$ belongs to  $\mathfrak{h}$-$\mathcal{D}(f,p(\cdot),q,\alpha)$.  
This shows that $f\in \mathcal{H}(p(\cdot),q,\alpha)$ and therefore
\begin{align*}
		\sigma_{p(\cdot),q,\alpha}(f) \leq \sigma_{p(\cdot),q,\alpha}(|f|).
\end{align*}
$\bullet$ Suppose that $f\in \mathcal{H}(p(\cdot),q,\alpha)$ and $\left\{\left(c_{n},\rho_{n},f_{n}\right)\right\}_{n\geq 1}$ is any element of $\mathfrak{h}$-$\mathcal{D}(f,p(\cdot),q,\alpha)$.
	It is clear that $|f| \leq \displaystyle\sum_{n\geq 1}|c_{n}|St_{\rho_{n}}^{(\alpha)}|f_{n}|=:g$. Then, by the lattice property, we get
	\begin{align*}
		\sigma_{p(\cdot),q,\alpha}(|f|)\leq \sigma_{p(\cdot),q,\alpha}(g)\leq \sum_{n\geq 1}|c_{n}|<\infty.
	\end{align*}
	This shows that
	\begin{align*}
		\sigma_{p(\cdot),q,\alpha}(|f|) \leq \sigma_{p(\cdot),q,\alpha}(f).
	\end{align*}
	The proof is complete.
\end{proof}

As an immediate consequence of Proposition \ref{P4.1}, we have what follows.
\begin{coro}
	$\mathcal{H}(p(\cdot),q,\alpha)$ is a  normed K\"{o}the space on $\mathbb{R}^d$ defined by $\sigma_{p(\cdot),q,\alpha}$.
\end{coro}

From what precedes, one can define a norm, denoted by $\| \cdot \|_{p(\cdot),q,\alpha}$, on $\mathcal{H}(p(\cdot),q,\alpha)$ as follows.

\begin{defi}
	For all  $f\in L^{p(\cdot)}_{\rm{loc}}$, we define
	\begin{align*}
		\|f\|_{\mathcal{H}(p(\cdot),q,\alpha)}= 	\sigma_{p(\cdot),q,\alpha}(|f|).
	\end{align*}
\end{defi}

In the sequel	$\mathcal{H}(p(\cdot),q,\alpha)$ will be always considered endowed with this norm.

We shall use the following theorem to prove our next result.

\begin{theo} \cite[Theorem 2.1]{Fei-Zim}. \label{th4.4}
	Let $\mathbf{B}$ be a Banach space and $\Phi = (\varphi_{j})_{j\in J}$ a (not necessarily countable) bounded family in $\mathbf{B}$. Define
\begin{align*}
	\mathcal{B} := \mathcal{B}_{\Phi} = \left\{f = \sum_{j\in J}a_{j}\varphi_{j}:\sum_{j\in J}|a_{j}|< \infty \right\},
\end{align*}
	and let 
	\begin{align*}
	\|f\|_{\mathcal{B}} = \inf \left\{\sum_{j\in J}|a_j| : f = \sum_{j\in J}a_j\varphi_j\right\}.
   \end{align*}
	Then $(\mathcal{B},\|\cdot\|_{\mathcal{B}})$ is a Banach space continuously embedded into \(\mathbf{B}\).
\end{theo}

\begin{prop} \label{p3.8}
Endowed with $\| \cdot \|_{\mathcal{H}(p(\cdot),q,\alpha)}$,	$\mathcal{H}(p(\cdot),q,\alpha)$ is a Banach space continuously embedded into $L^{\alpha}$.
\end{prop}

\begin{proof}
Since $ 1 \leq q \leq \alpha \leq p(\cdot)$, for any  
 $\varphi \in (L^{p(\cdot)},\ell^{q})$ and  $\rho > 0$, we have
 \begin{align*}
 	\|St_{\rho}^{(\alpha)} \varphi\|_{\alpha}=	\|\varphi\|_{\alpha}=\|\varphi\|_{\alpha,\alpha} \leq \|\varphi\|_{p(\cdot),\alpha} \leq \|\varphi\|_{p(\cdot),q} 
 \end{align*}
  and $\varphi\in L^{\alpha}$.
	The result follows from Theorem \ref{th4.4}, by using the Banach space $\mathbf{B} = L^{\alpha}$ and its bounded subset
\begin{align*}
	\Phi = \left\{St_{\rho}^{(\alpha)}\varphi : \rho >0 \text{ and } \|\varphi\|_{p(\cdot),q} \leq 1 \right\}.
\end{align*}
\end{proof} 

Point (1) of Proposition \ref{p2.6} and Proposition \ref{p3.8} yield obviously what follows.

\begin{coro} \label{c3.9}
	The function norm $\sigma_{p(\cdot),q,\alpha}$ has the Riesz-Fischer property.
\end{coro}

\begin{prop} \label{l3.10}
	The function norm 	$\sigma_{p(\cdot),q,\alpha}$ is saturated.
\end{prop}

\begin{proof}
	Let $E$ be any measurable subset of  $\mathbb{R}^{d}$ such that $|E|>0.$
	We set, for any integer $n\geq 1$, $Q_{n}=[-n,n)^{d}$.
	Note that, $\left\{E\cap Q_{n}\right\}_{n\geq 1}$ is a non-decreasing sequence of measurable subsets of  $E$ whose union is $E$. Consequently, we have $\displaystyle \lim_{n\rightarrow \infty}|E\cap Q_{n}|=|E|$. Thus, there exists a  positive integer $n_{0}$ such that $F_{n_{0}}:=E\cap Q_{n_{0}} \subset E$ and $0 <|F_{n_{0}}|\leq (2n_{0})^{d}$.
	Moreover, we have
\begin{align*}
		\|\chi_{F_{n_{0}}}\|_{p(\cdot),q} &\leq
	\left\{
	\begin{array}{rl}
		&\displaystyle \left(\sum_{k\in \mathbb{Z}^d, I^{1}_{k}\subset Q_{n_{0}}} \left\|\chi_{Q_{n_{0}}\cap I^{1}_{k}}\right\|_{p(\cdot)} \right)^{\frac{1}{q}}  \quad \mbox{ if } q<\infty, \\\\
		&\displaystyle\sup_{k\in \mathbb{Z}^d, I^{1}_{k}\subset Q_{n_{0}}} \left\|\chi_{Q_{n_{0}}\cap I^{1}_{k}} \right\|_{p(\cdot)}  \quad \quad \quad\; \mbox{ if } q =\infty.
	\end{array}
	\right. 
\end{align*}
It follows that
	\begin{align*}
		\|\chi_{F_{n_{0}}}\|_{p(\cdot),q}& \leq 	\left\{
	\begin{array}{rl}
		&\left(2n_{0} \right)^{\frac{d}{q}}  \quad \mbox{ if } q<\infty, \\\\
		&1 \quad \quad \quad \quad \mbox{if } q =\infty.
	\end{array}
	\right.
	\end{align*}
	This implies that 
	\begin{align*}
		\left\|\chi_{F_{n_{0}}} \right\|_{p(\cdot),q} \leq (2n_{0})^{\frac{d}{q}}<\infty.
	\end{align*}
Hence, we get
\begin{align*}
	\sigma_{p(\cdot),q,\alpha}\left(\chi_{F_{n_{0}}} \right)= \|\chi_{F_{n_{0}}}\|_{\mathcal{H}(p(\cdot),q,\alpha)} \leq 	\left\|\chi_{F_{n_{0}}} \right\|_{p(\cdot),q} <\infty.
\end{align*}  
Consequently, we may conclude that  $\sigma_{p(\cdot),q,\alpha}$ is saturated.   
\end{proof}


\subsection{Absolute continuity of the norm $\sigma_{p(\cdot),q,\alpha}$}

This subsection is devoted to study the absolute continuity of the norm $\sigma_{p(\cdot),q,\alpha}$. Here, we assume that $\Omega$ is a measurable subset of $\mathbb{R}^d$ with finite measure, that $\epsilon$ is any positive real number and $Q(0,r)$ denotes the cube centred at $0$ with side length $r$.

\begin{lem}\label{l3.11}
 Let $p(\cdot) \in \mathcal{P}(\mathbb{R}^d)$ and $f$ be any element of $L^{p(\cdot)}$. Then,
	\begin{enumerate}
		\item there exists $\delta_{\epsilon}>0$ such that $\|f\chi_{\Omega}\|_{p(\cdot)}<\epsilon$ whenever $|\Omega|<\delta_{\epsilon}$;
		\item there exists $R_{\epsilon}>0$ such that $\|f\chi_{\mathbb{R}^{d}\backslash Q(0,r) }\|_{p(\cdot)}<\epsilon$ whenever $R_{\epsilon}\leq r <\infty$.
	\end{enumerate}
\end{lem}

\begin{proof} 
	\begin{enumerate}
	\item Since $|\Omega|<\infty$, it follows from \eqref{eq2.5} that $\|f\chi_{\Omega}\|_{p(\cdot)} \leq C \,\|f\chi_{\Omega}\|_{p_{+}}$, where $C$ is a positive constant. Furthermore, since $p_{+}<\infty$, a well-known result in classical Lebesgue spaces shows that $\displaystyle \lim_{|\Omega|\rightarrow 0}\|f\chi_{\Omega}\|_{p_{+}}=0$. Therefore, we get $\displaystyle \lim_{|\Omega|\rightarrow 0}\|f\chi_{\Omega}\|_{p(\cdot)}=0$.
	\item We set, for all $r>0$, $f_{r}=f\chi_{Q(0,r)}$. Then, we have $\displaystyle \lim_{r \rightarrow \infty} f_{r} =f$ pointwise almost everywhere. It is clear that, for all $r>0$, $|f_{r}|\leq |f|$ and $f\chi_{\mathbb{R}^{d}\backslash Q(0,r) }=f-f_{r}$. 
	Therefore, using the dominated convergence theorem in variable Lebesgue spaces (see \cite[Theorem 2.62]{Cruz-Fior-2013}), we obtain
	\begin{align*}
		 \displaystyle \lim_{r \rightarrow \infty} \|f\chi_{\mathbb{R}^{d}\backslash Q(0,r)}\|_{p(\cdot)}= \displaystyle \lim_{r \rightarrow \infty} \|f-f_{r}\|_{p(\cdot)}=0.
	\end{align*} 
	\end{enumerate}

	The proof is complete.
\end{proof}

Lemma \ref{l3.11} may be extended to the setting of  variable-exponent Wiener amalgam spaces.
\begin{lem}\label{l3.12}
	Let $p(\cdot) \in \mathcal{P}(\mathbb{R}^d)$, $1\leq q \leq \infty$ and $f$ be any element of $\left(L^{p(\cdot)},\ell^{q}\right)$. Then,
	\begin{enumerate}
		\item there exists $\delta_{\epsilon}>0$ such that $\|f\chi_{\Omega}\|_{p(\cdot),q}<\epsilon$ whenever $|\Omega|<\delta_{\epsilon}$;
		\item there exists $R_{\epsilon}>0$ such that $\|f\chi_{\mathbb{R}^{d}\backslash Q(0,r)}\|_{p(\cdot),q}<\epsilon$ whenever $R_{\epsilon}\leq r <\infty$.
	\end{enumerate}
\end{lem}

\begin{proof}
	 Since $f \in \left(L^{p(\cdot)},\ell^{q}\right)$, there exists an integer $n_{0}\geq 1$ such that
\begin{align} \label{eq4.1}
	\left\{
	\begin{array}{rl}
	&	\displaystyle
		\sum_{k\in \mathbb{Z}^{d},|k|\geq n_{0}} \left\|f\chi_{I_{k}^{1}} \right\|_{p(\cdot)}^{q}<\left( \frac{\epsilon}{2}\right)^{q}  \qquad  \qquad \qquad \mbox{ if } q<\infty,\\
	&	\sup \left\{ \left\|f\chi_{I_{k}^{1}} \right\|_{p(\cdot)} : k\in \mathbb{Z}^{d},|k|\geq n_{0} \right\}<\frac{\epsilon}{2} \quad \; \; \mbox{ if } q=\infty.
	\end{array}
	\right.
\end{align}
Moreover, since the cardinality of the set $\left\{ k\in \mathbb{Z}^{d} : |k|\leq n_{0} \right\}$ is finite,
 Point $(1)$ of Lemma \ref{l3.11} asserts that there exists a real number $\delta_{\epsilon}>0$ such that, whenever $|\Omega|<\delta_{\epsilon}$, we have
 \begin{align} \label{eq4.2}
 	\left\{
 	\begin{array}{rl}
 		&	\displaystyle
 		\sum_{k\in \mathbb{Z}^{d},|k|< n_{0}} \left\|(f\chi_{\Omega})\chi_{I_{k}^{1}} \right\|_{p(\cdot)}^{q}<\left( \frac{\epsilon}{2}\right)^{q} \quad \quad \quad \quad \quad \quad \mbox{ if } q<\infty,\\
 		&	\max \left\{ \left\|(f\chi_{\Omega})\chi_{I_{k}^{1}} \right\|_{p(\cdot)} : k\in \mathbb{Z}^{d},|k|< n_{0} \right\}<\frac{\epsilon}{2} \quad  \; \mbox{ if } q=\infty.
 	\end{array}
 	\right.
 \end{align}
 Similarly, Point $(2)$ of Lemma \ref{l3.11} implies that there exists $R_{\epsilon}>0$ such that, whenever $R_{\epsilon}\leq r <\infty$, we have
 \begin{align} \label{eq4.3}
 	\left\{
 	\begin{array}{rl}
 		&	\displaystyle
 		\sum_{k\in \mathbb{Z}^{d},|k|< n_{0}} \left\|\left(f\chi_{\mathbb{R}^{d}\backslash Q(0,r)}\right)\chi_{I_{k}^{1}} \right\|_{p(\cdot)}^{q}<\left( \frac{\epsilon}{2}\right)^{q} \quad \quad \quad \quad \quad \quad \mbox{ if } q<\infty,\\
 		&	\max \left\{ \left\|\left(f\chi_{\mathbb{R}^{d}\backslash Q(0,r)}\right)\chi_{I_{k}^{1}} \right\|_{p(\cdot)} : k\in \mathbb{Z}^{d},|k|< n_{0} \right\}<\frac{\epsilon}{2} \quad  \; \mbox{ if } q=\infty.
 	\end{array}
 	\right.
 \end{align}
 A combination of \eqref{eq4.1} and \eqref{eq4.2} proves Point $(1)$. Likewise \eqref{eq4.1} and \eqref{eq4.3} prove Point $(2)$.
\end{proof}

Next, Lemma \ref{l3.12} leads to the following proposition. 

\begin{prop}\label{prop3.13}
	Let $p(\cdot) \in \mathcal{P}(\mathbb{R}^d)$ and $f$ be any element of $\mathcal{H}(p(\cdot),q,\alpha)$. Then,
\begin{enumerate}
\item there exists $\delta_{\epsilon}>0$ such that $\|f\chi_{\Omega}\|_{\mathcal{H}(p(\cdot),q,\alpha)}<\epsilon$ whenever $|\Omega|<\delta_{\epsilon}$,
\item there exists $R_{\epsilon}>0$ such that $\|f\chi_{\mathbb{R}^{d}\backslash Q(0,r) }\|_{\mathcal{H}(p(\cdot),q,\alpha)}<\epsilon$ whenever $R_{\epsilon}\leq r <\infty$.
\end{enumerate}
\end{prop}

\begin{proof}
	If $f=0$ then the desired result is obvious. Thus we assume that $f \neq 0$. Let  $\left\{ \left( c_{n},\rho_{n},f_{n}\right)\right\}_{n\geq 1}$ be any element of $\mathfrak{h}$-$\mathcal{D}(f,p(\cdot),q,\alpha)$. There exists an integer $n_{0}\geq 1$  such that $\displaystyle\sum_{n\geq n_{0}+1}|c_{n}|< \frac{\epsilon}{2}$. Let us set 	$h_{n_{0}}=\displaystyle\sum_{n= 1}^{n_{0}}c_{n}St_{\rho_{n}}^{(\alpha)}f_{n}$.\\
$\bullet$	We have
\begin{align*}
	\|f\chi_{\Omega}\|_{\mathcal{H}(p(\cdot),q,\alpha)} &\leq
 \|f-h_{n_{0}}\|_{\mathcal{H}(p(\cdot),q,\alpha)} + \|h_{n_{0}}\chi_{\Omega}\|_{\mathcal{H}(p(\cdot),q,\alpha)}\\
&\leq \displaystyle\sum_{n\geq n_{0}+1}|c_{n}|+\|h_{n_{0}}\chi_{\Omega}\|_{\mathcal{H}(p(\cdot),q,\alpha)}\\
&< \frac{\epsilon}{2}+\|h_{n_{0}}\chi_{\Omega}\|_{\mathcal{H}(p(\cdot),q,\alpha)}.
\end{align*}
	Furthermore, $h_{n_{0}}\chi_{\Omega}$ may be written as follows:
	\begin{align*}
		h_{n_{0}}\chi_{\Omega}&=\displaystyle\sum_{n= 1}^{n_{0}}c_{n}\left(St_{\rho_{n}}^{(\alpha)}f_{n}\right)\chi_{\Omega}= \displaystyle\sum_{n= 1}^{n_{0}}c_{n}St_{\rho_{n}}^{(\alpha)}\left(f_{n}\chi_{\rho_{n}^{-1}\Omega}\right)\\
		&=\displaystyle \sum_{n= 1}^{n_{0}} c_{n} \left\|f_{n}\chi_{\rho_{n}^{-1}\Omega}\right\|_{p(\cdot),q} St_{\rho_{n}}^{(\alpha)}\left(\left\|f_{n}\chi_{\rho_{n}^{-1}\Omega}\right\|_{p(\cdot),q}^{-1}f_{n}\chi_{\rho_{n}^{-1}\Omega}\right).
	\end{align*}
This implies that
\begin{align*}
\left\|h_{n_{0}}\chi_{\Omega} \right\|_{\mathcal{H}(p(\cdot),q,\alpha)} \leq \displaystyle \sum_{n= 1}^{n_{0}} c_{n} \left\|f_{n}\chi_{\rho_{n}^{-1}\Omega}\right\|_{p(\cdot),q}
\end{align*}
and therefore, we get
\begin{align}\label{eq.3.3}
		\|f\chi_{\Omega}\|_{\mathcal{H}(p(\cdot),q,\alpha)}  < \frac{\epsilon}{2}+\displaystyle \sum_{n= 1}^{n_{0}} c_{n} \left\|f_{n}\chi_{\rho_{n}^{-1}\Omega}\right\|_{p(\cdot),q}.
\end{align}
Similarly, we have
\begin{align}\label{eq.3.4}
	\|f\chi_{\mathbb{R}^{d}\backslash Q(0,r)}\|_{\mathcal{H}(p(\cdot),q,\alpha)}  < \frac{\epsilon}{2}+\displaystyle \sum_{n= 1}^{n_{0}} c_{n} \left\|f_{n}\chi_{\rho_{n}^{-1}\left(\mathbb{R}^{d}\backslash Q(0,r)\right)}\right\|_{p(\cdot),q}.
\end{align}
$\bullet$ By Lemma \ref{l3.12}, for any integer $n\geq 1,$ there are two positive real numbers $\delta_{\epsilon}^{n}$ and $R_{\epsilon}^{n}$ such that,
	\begin{align*}
	\left\{
	\begin{array}{rl}
	 &\left\|f_{n}\chi_{\Omega}\right\|_{p(\cdot),q}<\frac{\epsilon}{2n_{0}\left(|c_{n}|+1\right)} \quad \quad \quad \quad \mbox{ whenever } |\Omega|<\delta_{\epsilon}^{n}, \\
	&	\left\|f_{n}\chi_{\mathbb{R}^{d}\backslash Q(0,r)}\right\|_{p(\cdot),q} <\frac{\epsilon}{2n_{0}\left( |c_{n}|+1\right)} \quad \mbox{ whenever } R_{\epsilon}^{n} \leq r <\infty.
	\end{array}
	\right.
   \end{align*}
	Let us set  $\delta_{\epsilon} = \displaystyle \min_{1\leq n\leq n_{0}}\rho_{n}^{d}\delta_{\epsilon}^{n}$  and $R_{\epsilon}=\displaystyle \max_{1\leq n\leq n_{0}} \rho_{n}R_{\epsilon}^{n}$. Then, for any integer $n$ satisfying $1\leq n\leq n_{0}$, we have
	\begin{align*}
		\left\{
		\begin{array}{rl}
			&\left\|f_{n}\chi_{\rho_{n}^{-1} \Omega}\right\|_{p(\cdot),q}<\frac{\epsilon}{2n_{0}\left(|c_{n}|+1\right)} \quad \quad \quad \quad \;\, \mbox{ whenever } |\Omega|<\delta_{\epsilon}, \\
			&	\left\|f_{n}\chi_{\rho_{n}^{-1}\left(\mathbb{R}^{d}\backslash Q(0,r)\right)}\right\|_{p(\cdot),q} <\frac{\epsilon}{2n_{0}\left( |c_{n}|+1\right)} \quad \mbox{whenever } R_{\epsilon} \leq r <\infty.
		\end{array}
		\right.
	\end{align*}
The above inequalities combined with \eqref{eq.3.3} and \eqref{eq.3.4} yield 
\begin{align*}
		\left\{
	\begin{array}{rl}
		&\|f\chi_{\Omega}\|_{\mathcal{H}(p(\cdot),q,\alpha)}  < \epsilon \quad \quad \quad \quad \mbox{ whenever } |\Omega|<\delta_{\epsilon}, \\
		&\|f\chi_{\mathbb{R}^{d}\backslash Q(0,r)}\|_{\mathcal{H}(p(\cdot),q,\alpha)}  < \epsilon \quad \mbox{ whenever } R_{\epsilon} \leq r <\infty.
	\end{array}
	\right.
\end{align*}
The proof is complete.
\end{proof}

As a consequence of Proposition \ref{prop3.13}, we can state the main result of this subsection.

\begin{theo}\label{th3.14}
	Suppose that $p(\cdot) \in \mathcal{P}(\mathbb{R}^d)$. Then the function norm $\sigma_{p(\cdot),q,\alpha}$ is absolutely continuous.
\end{theo}

\begin{proof}
	Let $f$ be any element of $\mathcal{H}(p(\cdot),q,\alpha)$ and $\{ E_{n}\}_{n\geq 1}$ be a non-increasing sequence of measurable subsets of $\mathbb{R}^{d}$ satisfying $\left|\displaystyle\bigcap_{n\geq 1}E_{n}\right|=0$. Using Proposition \ref{prop3.13} and arguing as in the proof of \cite[Proposition 3.18]{Dia-Fof-2021} we get 
	$\displaystyle \lim_{n\rightarrow \infty}\|f\chi_{E_{n}}\|_{\mathcal{H}(p(\cdot),q,\alpha)} = 0.$
	Consequently, Point (4) of Proposition \ref{p2.6} asserts that $f$ is of absolutely continuous norm. This implies that
	 $\sigma_{p(\cdot),q,\alpha}$ is an absolutely continuous norm. 
\end{proof}


\subsection{The Fatou property for $\mathcal{H}(p(\cdot),q,\alpha)$} \label{Subsec3.3}

 This subsection focuses on the study of the Fatou property in $\mathcal{H}(p(\cdot),q,\alpha)$. Before, we need some preparatory results.

The following inequality holds for dilation in variable Lebesgue spaces (see \cite{Cruz-Fior-2013}).
\begin{lem} \label{lem3.10} 
Let $p(\cdot)\in \mathcal{P}(\mathbb{R}^d)$ and $\rho \in (0,\infty)$. Then, for any $f \in L^{p(\cdot)}$, we have
\begin{align} \label{eq3.3}
	\left\|f(\rho^{-1}\cdot) \right\|_{p(\cdot)} \leq  \rho^{\frac{d}{N_{p,\rho}}}\|f \|_{p(\cdot)},
\end{align}
where 
\begin{align*}
N_{p,\rho} = \left\{
\begin{array}{rl}
	&p_{+} \;\; \mbox{ if } \rho \leq 1,\\
	&p_{-} \;\; \mbox{ if } \rho > 1.
\end{array}
\right.
\end{align*}
\end{lem}

As a consequence of Lemma \ref{lem3.10} we have what follows.
\begin{lem} \label{lem3.11}
	Let $p(\cdot)\in \mathcal{P}(\mathbb{R}^d)$, $(r,\rho) \in (0,\infty)^{2}$, $\alpha \in (0,\infty]$ and $1\leq q \leq \infty$. Then, for any $f \in L^{p(\cdot)}_{\rm loc}$, we have
	\begin{align} \label{eq3.4}
		_{r} \left\|St_{\rho}^{(\alpha)}f \right\|_{p(\cdot),q} \leq  \rho^{d \left(\frac{1}{N_{p,\rho}}-\frac{1}{\alpha}\right)} {_{r\rho^{-1}}\|f\|_{p(\cdot),q}}.
	\end{align}
\end{lem}
\begin{proof}
	Suppose that $f \in L^{p(\cdot)}_{\rm loc}$ and $k\in \mathbb{Z}^{d}$. For any $x\in \mathbb{R}^{d}$, we have
	\begin{align*}
	\left(St_{\rho}^{(\alpha)}f\right)\chi_{I_{k}^{r}}(x) &= \rho^{-\frac{d}{\alpha}}f(\rho^{-1}x)\chi_{I_{k}^{r}}(x)= \rho^{-\frac{d}{\alpha}}f(\rho^{-1}x)\chi_{I_{k}^{r\rho^{-1}}}(\rho^{-1}x) \\
	&= \rho^{-\frac{d}{\alpha}}\left(f\chi_{I_{k}^{r\rho^{-1}}}\right)(\rho^{-1}x).
	\end{align*}
		Therefore, taking the norm $\|\cdot\|_{p(\cdot)}$ of both sides of the above equality and using \eqref{eq3.3}, we get
		\begin{align*}
			\left\|\left(St_{\rho}^{(\alpha)}f\right)\chi_{I_{k}^{r}} \right\|_{p(\cdot)} \leq  \rho^{d \left(\frac{1}{N_{p,\rho}}-\frac{1}{\alpha}\right)} \left\|f\chi_{I_{k}^{r\rho^{-1}}} \right\|_{p(\cdot)}.
		\end{align*}
		Finally, we obtain the desired inequality \eqref{eq3.4} by taking the $\ell^{q}$-norm over $k\in \mathbb{Z}^{d}$.
\end{proof}

We shall also use the following result concerning dyadic cubes in $\mathbb{R}^{d}$.

\begin{lem} \cite[Lemma 2.1]{Dia-Nag-2024}. \label{lem3.12}
	Let $(m,\rho)$ be an element of $\mathbb{Z}\times \mathbb{R}_{+}^{*}$ such that  $2^{m}\leq \rho< 2^{m+1}$. Then the number of elements of the set $\left\{ \ell \in\mathbb{Z}^{d} \;:\; I^{2^{m}}_{\ell}\cap I^{\rho}_{k}\neq\varnothing \right\}$ is at most $3^{d}$ and 
	that of the set $\left\{ k \in\mathbb{Z}^{d} \;:\; I^{2^{m}}_{\ell}\cap I^{\rho}_{k}\neq\varnothing \right\}$ is at most $2^{d}$.
\end{lem}

Our next result concerns the dyadic decomposition of an element of $\mathcal{H}(p(\cdot),q,\alpha)$, which will play an important role in the sequel. Note that this decomposition can also be used to prove Corollary \ref{c3.9} directly.

\begin{prop} \label{p3.13}
	Let $p(\cdot)\in \mathcal{P}(\mathbb{R}^d)$. Suppose that $f \in \mathcal{H}(p(\cdot),q,\alpha)$ and  $\left\{(c_{n},\rho_{n},f_{n})\right\}_{n\geq 1}$ is any element of  $\mathfrak{h}$-$\mathcal{D}(f,p(\cdot),q,\alpha)$.
	Then $f$ admits the dyadic decomposition  $f=\displaystyle \sum_{m\in \mathbb{Z}}\lambda_{m}St_{2^{m}}^{(\alpha)}\psi_{m}$, where $\left\{(\lambda_{m},\psi_{m})\right\}_{m\in \mathbb{Z}}$ is a sequence of elements of $[0, \infty)\times (L^{p(\cdot)}, \ell^{q})$ such that 
	\begin{align*}
		\left\{
		\begin{array}{rl}
			&\displaystyle \sum_{m\in \mathbb{Z}}\lambda_{m}= 2^{\frac{d}{q'}}3^{\frac{d}{q}}\displaystyle \sum_{n\geq 1}|c_{n}|, \\
			&\|\psi_{m}\|_{p(\cdot),q} \leq 1, \quad m\in\mathbb{Z}.
		\end{array}
		\right.
	\end{align*} 
\end{prop}

\begin{proof}
Inequality \eqref{eq3.4} implies that, for any integer $n\geq 1$,
	$\varphi_{n}:= St_{\rho_{n}}^{(\alpha)}f_{n}$ enjoys 
	\begin{align}\label{eq3.5}
		{_{\rho_{n}}\|\varphi_{n}\|_{p(\cdot),q}} \leq \rho_{n}^{d \left(\frac{1}{N_{p,\rho}}-\frac{1}{\alpha}\right)} \|f_{n}\|_{p(\cdot),q} \leq \rho_{n}^{d \left(\frac{1}{N_{p,\rho}}-\frac{1}{\alpha}\right)}.
	\end{align}
 Set, for any $m$ in $\mathbb{Z},$
	$$ \Gamma(m)=\left\{ n \in \mathbb{N}\;:\;2^{m}\leq \rho_{n}< 2^{m+1} \right\}.$$
	Let $m\in \mathbb{Z}$ and  $n$ be in $ \Gamma(m)$. Then, for all $k \in  \mathbb{Z}^{d}$, H\"{o}lder's inequality and Lemma \ref{lem3.12} lead to
	\begin{align*}
		\left\|\varphi_{n}\chi_{I_{k}^{2^{m}}} \right\|_{p(\cdot)} 
		&\leq \, \sum_{\ell \in \mathbb{Z}^{d}}	\left\|\varphi_{n}\chi_{I_{k}^{2^{m}}\cap I^{\rho_{n}}_{\ell}} \right\|_{p(\cdot)}	\\
		&\leq 2^{\frac{d}{q'}} \left(\sum_{\ell \in \mathbb{Z}^{d}}	\|\varphi_{n}\chi_{I_{k}^{2^{m}}\cap I^{\rho_{n}}_{\ell}}\|_{p(\cdot)}^{q} \right)^{ \frac{1}{q}}.
	\end{align*}
	Therefore, we have
	\begin{align*}
		_{2^{m}}\|\varphi_{n}\|_{p(\cdot),q}
		&= \left\| \left\{\left \|\varphi_{n}\chi_{I^{2^{m}}_{k}}\right\|_{p(\cdot)} \right\}_{k\in \mathbb{Z}^d}\right\|_{\ell^{q}}
		\leq 2^{\frac{d}{q'}} \left(\sum_{k\in \mathbb{Z}^d} \left(\sum_{\ell \in \mathbb{Z}^{d}}	\|\varphi_{n}\chi_{I_{\ell}^{\rho_{n}}\cap I^{2^{m}}_{k}}\|_{p(\cdot)}^{q}\right) \right)^{\frac{1}{q}}\\
		&=2^{\frac{d}{q'}} \left(\sum_{\ell\in \mathbb{Z}^d} \left( \sum_{k \in \mathbb{Z}^{d}}	\|\varphi_{n}\chi_{I_{\ell}^{\rho_{n}}\cap I^{2^{m}}_{k}}\|_{p(\cdot)}^{q}\right) \right)^{\frac{1}{q}}\\
		&  \leq	2^{\frac{d}{q'}}  3^{\frac{d}{q}}\,\left(\sum_{\ell\in \mathbb{Z}^d}	\|\varphi_{n}\chi_{ I^{\rho_{n}}_{\ell}}\|_{p(\cdot)}^{q} \right)^{\frac{1}{q}} (\mbox{by Lemma \ref{lem3.12}}).\\
		&=2^{\frac{d}{q'}}  3^{\frac{d}{q}} \, _{\rho_{n}}\|\varphi_{n}\|_{p(\cdot),q}.
	\end{align*}
	Hence, from \eqref{eq3.5}, we obtain
	\begin{align} \label{eq3.6}
			_{2^{m}}\|\varphi_{n}\|_{p(\cdot),q} \leq 2^{\frac{d}{q'}}  3^{\frac{d}{q}} \rho_{n}^{d \left(\frac{1}{N_{p,\rho}}-\frac{1}{\alpha}\right)}.
	\end{align}
	Furthermore, we have 
	\begin{align*}
		f=\displaystyle \sum_{n\geq 1}c_{n}\varphi_{n}=\displaystyle \sum_{m\in\mathbb{Z}}\;\;\sum_{n\in \Gamma(m)}c_{n}\varphi_{n}.
	\end{align*}
	Therefore, we can write
	\begin{align*}
		f=\displaystyle \sum_{m\in\mathbb{Z}}\left( 2^{\frac{d}{q'}} 3^{\frac{d}{q}}\sum_{n\in \Gamma(m)}|c_{n}|\right)\left[ \left(  2^{\frac{d}{q'}}3^{\frac{d}{q}}\sum_{n\in \Gamma(m)}|c_{n}|\right)^{-1} \sum_{n\in \Gamma(m)}c_{n}\;\varphi_{n}\right]= \sum_{m\in\mathbb{Z}}\lambda_{m}St_{2^{m}}^{(\alpha)}\psi_{m},
	\end{align*}
	where, for any $m\in \mathbb{Z},$
	\begin{align*}
	\lambda_{m}=2^{\frac{d}{q'}} 3^{\frac{d}{q}}\displaystyle\sum_{n\in \Gamma(m)}|c_{n}| \mbox{ and } \psi_{m}=St_{2^{-m}}^{(\alpha)}\left(\lambda_{m}^{-1}\displaystyle\sum_{n\in \Gamma(m)}c_{n}\;\varphi_{n}\right).
	\end{align*}
	Obviously, from the definition, every $\lambda_{m} (m\in\mathbb{Z})$ is nonnegative and
	$$\displaystyle \sum_{m\in \mathbb{Z}}\lambda_{m}=2^{\frac{d}{q'}} 3^{\frac{d}{q}}\displaystyle \sum_{n\geq 1}|c_{n}|<\infty.$$
	Moreover, for any integer $m$, we have the following estimates by using \eqref{eq3.4} with $\rho=2^{-m}$, \eqref{eq3.6} and the definition of $\Gamma(m)$:
	\begin{align*}
		\|\psi_{m}\|_{p(\cdot),q} 
		&\leq  \lambda_{m}^{-1}\;2^{md \left(\frac{1}{\alpha}-\frac{1}{N_{p,\rho}}\right)} \quad  _{_{_{_{_{_{_{{_{_{2^{m}}}}}}}}}}}\left\|\displaystyle \sum_{n\in \Gamma(m)}c_{n}\;\varphi_{n}\right\|_{p(\cdot),q} \\
		&\leq \lambda_{m}^{-1}\;2^{md \left(\frac{1}{\alpha}-\frac{1}{N_{p,\rho}}\right)}\displaystyle\sum_{n\in \Gamma(m)}|c_{n}|\;_{_{{_{2^{m}}}}\|\varphi_{n}\|_{p(\cdot),q}}\\
		&\leq \lambda_{m}^{-1}\;2^{md \left(\frac{1}{\alpha}-\frac{1}{N_{p,\rho}}\right)}\;2^{\frac{d}{q'}}  3^{\frac{d}{q}} \displaystyle\sum_{n\in \Gamma(m)} |c_{n}|\,\rho_{n}^{d \left(\frac{1}{N_{p,\rho}}-\frac{1}{\alpha}\right)}\\
		&\leq \lambda_{m}^{-1}\left(2^{\frac{d}{q'}}  3^{\frac{d}{q}}\displaystyle\sum_{n\in \Gamma(m)}|c_{n}|\right)= 1.
	\end{align*}
	This ends the proof.
\end{proof}

We shall use the following notations.
\begin{nota} \label{not3.14}
Let	$\left(m_{0},\ell_{0}\right)$ be any element of $\mathbb{Z}\times \mathbb{Z}^{d}$. Let us fix two integers $m_{1}$ and
$m_{2}$ such that $m_{1}\leq 0$, $m_{2}\geq 0$ and $m_{1}<m_{0}<m_{2}$.
We set $\mathcal{Q}_{0}:=I_{\ell_{0}}^{2^{m_{0}}}$ ($\mathcal{Q}_{0}$ is a dyadic cube in $\mathbb{R}^{d}$) and denote 
	\begin{align*}
		&\mathcal{D}_{1}(\mathcal{Q}_{0})=\left\{(m,\ell)\in\mathbb{Z}\times \mathbb{Z}^{d} : I_{\ell}^{2^{m}}\cap \mathcal{Q}_{0}\neq\varnothing \mbox{ and } m\leq m_{1}\right\},\\
		&\mathcal{D}_{2}(\mathcal{Q}_{0})=\left\{(m,\ell)\in\mathbb{Z}\times \mathbb{Z}^{d} : I_{\ell}^{2^{m}}\cap \mathcal{Q}_{0}\neq\varnothing \mbox{ and } m_{2}\leq m \right\},\\
		&\mathcal{D}_{3}(\mathcal{Q}_{0})=\left\{(m,\ell)\in\mathbb{Z}\times \mathbb{Z}^{d} : I_{\ell}^{2^{m}}\cap \mathcal{Q}_{0}\neq\varnothing \mbox{ and } m_{1}<m<m_{2} \right\}.
	\end{align*}
\end{nota}

\begin{lem}\label{lem3.15} 
	Let $p(\cdot)\in \mathcal{P}(\mathbb{R}^d)$ and  $\left\{\left(\lambda_{m},\phi_{m}\right)\right\}_{m\in\mathbb{Z}}$ be a sequence of elements of $[0,\infty)\times L^{p(\cdot)}_{\rm loc}$. Suppose that $\displaystyle \sum_{m\in \mathbb{Z}}\lambda_{m}<\infty$  and, for any $m\in\mathbb{Z}$, ${_{2^{m}}\|\phi_{m}\|_{p(\cdot),q}}\leq 2^{md\left(\frac{1}{N_{p,2^{m}}}-\frac{1}{\alpha} \right)}$, where $N_{p,2^{m}}$ is defined as in Lemma \ref{lem3.10}. We set $A_{p}=1+\frac{1}{p_{-}} -\frac{1}{p_{+}}$. Then, we have
	\begin{align}
		&\displaystyle \sum_{(m,\ell)\in \mathcal{D}_{1}(\mathcal{Q}_{0})}\int_{\mathcal{Q}_{0}}\left| \lambda_{m} \phi_{m}(x)\chi_{I_{\ell}^{2^{m}}}(x)\right|dx \leq  A_{p}\;2^{\frac{m_{0}d}{q^{\prime}}+m_{1}d\left( \frac{1}{\alpha^{\prime}}-\frac{1}{q^{\prime}}\right)}\displaystyle\sum _{m\leq m_{1}}\lambda_{m},  \label{eq3.7} \\
		& \displaystyle \sum_{(m,\ell)\in \mathcal{D}_{2}(\mathcal{Q}_{0})}\int_{\mathcal{Q}_{0}}\left| \lambda_{m} \phi_{m}(x)\chi_{I_{\ell}^{2^{m}}}(x)\right|dx \leq    A_{p}\;2^{\frac{m_{0}d}{\left(N_{p,2^{m_{0}}}\right)^{\prime}}+m_{2}d\left(\frac{1}{p_{-}}-\frac{1}{\alpha} \right)} \displaystyle\sum _{m\geq m_{2}} \lambda_{m}, \label{eq3.8} \\ 
		& \displaystyle \sum_{(m,\ell)\in \mathcal{D}_{3}(\mathcal{Q}_{0})}\int_{\mathcal{Q}_{0}}\left| \lambda_{m} \phi_{m}(x)\chi_{I_{\ell}^{2^{m}}}(x)\right|dx \leq  A_{p} \, 2^{\frac{m_{2}d}{\alpha^{\prime}}}\displaystyle\sum _{m_{1}<m< m_{2}}\lambda_{m}. \label{eq3.9}
	\end{align} 
\end{lem}

\begin{proof}
	$\bullet$ We first establish \eqref{eq3.7}.
	\begin{align*}
		&\displaystyle \sum_{(m,\ell)\in \mathcal{D}_{1}(\mathcal{Q}_{0})}\int_{\mathcal{Q}_{0}}\left| \lambda_{m} \phi_{m}(x)\chi_{I_{\ell}^{2^{m}}}(x)\right|dx 
		=\displaystyle\sum _{m\leq m_{1}}\lambda_{m}\sum_{I_{\ell}^{2^{m}}\subset \mathcal{Q}_{0}}\int_{I_{\ell}^{2^{m}}}\left| \phi_{m}(x)\right|dx\\
		& \leq \displaystyle\sum _{m\leq m_{1}}\lambda_{m}\sum_{I_{\ell}^{2^{m}}\subset \mathcal{Q}_{0}} A_{p} \,\|\chi_{I_{\ell}^{2^{m}}}\|_{p^{\prime}(\cdot)} \;\|\phi_{m}\chi_{I_{\ell}^{2^{m}}}\|_{p(\cdot)} \quad \mbox{ (by \eqref{eq2.4})}\\
		& \leq A_{p} \displaystyle\sum _{m\leq m_{1}}\lambda_{m}\; 2^{\frac{md}{(N_{p,2^{m}})^{\prime}}}\sum_{I_{l}^{2^{m}}\subset \mathcal{Q}_{0}}\|\phi_{m}\chi_{I_{\ell}^{2^{m}}}\|_{p(\cdot)}
		\leq  A_{p} \displaystyle \sum _{m\leq m_{1}}\lambda_{m}\;2^{\frac{md}{(N_{p,2^{m}})^{\prime}}} 2^{\frac{\left(m_{0}-m\right)d}{q^{\prime}}} {_{2^{m}}\|\phi_{m}\|_{p(\cdot),q}}\\ 
		&\leq A_{p} \displaystyle \sum _{m\leq m_{1}}\lambda_{m}\;2^{\frac{md}{(N_{p,2^{m}})^{\prime}}} 2^{\frac{\left(m_{0}-m\right)d}{q^{\prime}}} 2^{md\left(\frac{1}{N_{p,2^{m}}}-\frac{1}{\alpha} \right)}
		\leq  A_{p} \; 2^{\frac{m_{0}d}{q^{\prime}}+m_{1}d\left( \frac{1}{\alpha'}-\frac{1}{q^{\prime}}\right)}\displaystyle\sum _{m\leq m_{1}}\lambda_{m}.
	\end{align*}
    $\bullet$ Next, we prove \eqref{eq3.8}. We have
	\begin{align*}
		\displaystyle \sum_{(m,\ell)\in \mathcal{D}_{2}(\mathcal{Q}_{0})}\int_{\mathcal{Q}_{0}}\left| \lambda_{m} \phi_{m}(x)\chi_{I_{\ell}^{2^{m}}}(x)\right|dx = \displaystyle\sum _{m\geq m_{2}}\lambda_{m}\sum_{I_{\ell}^{2^{m}}\subset \mathcal{Q}_{0}}\int_{I_{\ell}^{2^{m}}}\left| \phi_{m}(x)\right|dx.
	\end{align*}
Note that, for any integer $m\geq m_{0}$, there is one and only one element $\ell_{m}$ of $\mathbb{Z}^{d}$ such that $ I_{\ell_{m}}^{2^{m}}\cap \mathcal{Q}_{0} \neq \varnothing$. Moreover $\mathcal{Q}_{0}$ is included in $ I_{\ell_{m}}^{2^{m}}$. Therefore
	\begin{align*}
		& \displaystyle \sum_{(m,\ell)\in \mathcal{D}_{2}(\mathcal{Q}_{0})}\int_{\mathcal{Q}_{0}}\left| \lambda_{m} \phi_{m}(x)\chi_{I_{\ell}^{2^{m}}}(x)\right|dx 
		=\displaystyle\sum _{m\geq m_{2}} \lambda_{m} \int_{\mathcal{Q}_{0}}\left| \phi_{m}(x)\chi_{I_{\ell_{m}}^{2^{m}}}(x)\right|dx\\
		&\leq \displaystyle\sum _{m\geq m_{2}} \lambda_{m} \, A_{p} \,\|\chi_{\mathcal{Q}_{0}}\|_{p^{\prime}(\cdot)} \;\|\phi_{m}\chi_{I_{\ell_{m}}^{2^{m}}}\|_{p(\cdot)}
		\leq  A_{p} \,2^{\frac{m_{0}d}{\left(N_{p,2^{m_{0}}}\right)^{\prime}}}\displaystyle\sum _{m\geq m_{2}}\lambda_{m}\;{_{2^{m}}\|\phi_{m}\|_{p(\cdot),q}}\\
		&\leq  A_{p} \,2^{\frac{m_{0}d}{\left(N_{p,2^{m_{0}}}\right)^{\prime}}} \displaystyle\sum _{m\geq m_{2}} \lambda_{m}\;2^{md\left(\frac{1}{N_{p,2^{m}}}-\frac{1}{\alpha} \right)}  
	 \leq  A_{p} \,2^{\frac{m_{0}d}{\left(N_{p,2^{m_{0}}}\right)^{\prime}}} \displaystyle\sum _{m\geq m_{2}} \lambda_{m}\;2^{m_{2}d\left(\frac{1}{p_{-}}-\frac{1}{\alpha} \right)} \\
	 &\leq  A_{p}\;2^{\frac{m_{0}d}{\left(N_{p,2^{m_{0}}}\right)^{\prime}}+m_{2}d\left(\frac{1}{p_{-}}-\frac{1}{\alpha} \right)} \displaystyle\sum _{m\geq m_{2}} \lambda_{m}.
	\end{align*}
	$\bullet$ Finally, we establish \eqref{eq3.9}.
	\begin{align*}
		& \displaystyle \sum_{(m,\ell)\in \mathcal{D}_{3}(\mathcal{Q}_{0})}\int_{\mathcal{Q}_{0}}\left| \lambda_{m} \phi_{m}(x)\chi_{I_{\ell}^{2^{m}}}(x)\right|dx
		\leq \displaystyle \sum_{(m,\ell)\in \mathcal{D}_{3}} \lambda_{m} \, A_{p} \,\|\chi_{\mathcal{Q}_{0}\cap I_{\ell}^{2^{m}} }\|_{p^{\prime}(\cdot)} \;\|\phi_{m}\chi_{I_{\ell}^{2^{m}}}\|_{p(\cdot)}\\
		&\leq A_{p} \displaystyle \sum_{m_{1}<m<m_{2}}\lambda_{m}\; 2^{\frac{md}{(N_{p,2^{m}})^{\prime}}} {_{2^{m}}\|\phi_{m}\|_{p(\cdot),q}}
		\leq A_{p} \displaystyle \sum_{m_{1}<m<m_{2}}\lambda_{m}\; 2^{\frac{md}{(N_{p,2^{m}})^{\prime}}}  2^{md\left(\frac{1}{N_{p,2^{m}}}-\frac{1}{\alpha} \right)} \\
		&\leq A_{p} \displaystyle \sum_{m_{1}<m<m_{2}} \lambda_{m} \,2^{\frac{md}{\alpha^{\prime}}} \leq A_{p} \,2^{\frac{m_{2}d}{\alpha^{\prime}}} \displaystyle \sum_{m_{1}<m<m_{2}} \lambda_{m}.
	\end{align*}
	The proof is complete.
\end{proof}

We next establish the following proposition, which will play an important role in the proof of the main result of this subsection.

\begin{prop} \label{p3.16}
	Let $p(\cdot)\in \mathcal{P}(\mathbb{R}^d)$, $1\leq q <\alpha < p(\cdot)$ and
	$\left(f_{n}\right)_{n\geq 1}$ be a non-decreasing  sequence of nonnegative elements of $\mathcal{H}(p(\cdot),q,\alpha)$.  Then $f:=\displaystyle \sup_{n\geq 1}f_{n}$ belongs to $\mathcal{H}(p(\cdot),q,\alpha)$. 
\end{prop}

	 \begin{proof}
	 	$\bullet$ Without loss of generality, we may assume that, for any $n\geq 1$,  $\| f_{n}\|_{\mathcal{H}(p(\cdot),q,\alpha)}\leq 1$.
		 From Proposition \ref{p3.13}, every $f_{n}(n\geq 1)$ can be decomposed as follows:
		\begin{align*}
		f_{n}=\displaystyle \sum_{m\in \mathbb{Z}}\lambda^{n}_{m}St_{2^{m}}^{(\alpha)}\psi^{n}_{m},
		\end{align*}
		where $ \left\{\left( \lambda_{m}^{n},\psi_{m}^{n}\right) \right\}_{m\in\mathbb{Z}}$ is a sequence of elements of $[0,\infty)\times (L^{p(\cdot)}, \ell^{q})$  such that
		\begin{align*}
		\displaystyle \sum_{m\in \mathbb{Z}}\lambda^{n}_{m}\leq 2^{\frac{d}{q^{\prime}}+1}\,3^{\frac{d}{q}}
		\end{align*}
		and $\|\psi^{n}_{m}\|_{p(\cdot),q}\leq 1$ for all $m \in \mathbb{Z}$.
		  Since, for all $n\geq 1$, $f_{n}$ is nonnegative, $\psi_{m}^{n}$  also  may be choosen nonnegative.
		Note that, for any $(m,\ell)$ in $\mathbb{Z}\times\mathbb{Z}^{d},$ 
		$\left\{\lambda_{m}^{n} : n\geq 1\right\}$ and 	$\left\{\psi_{m}^{n}\chi_{I_{\ell}^{2^{m}}} : n\geq 1\right\}$ 
		are bounded subsets of $[0,\infty)$ and $L^{p(\cdot)}$, respectively.
      Therefore, using a diagonalization argument, we can construct a subsequence $\left(f_{n_{j}}\right)_{j\geq 1}$ of $\left(f_{n}\right)_{n\geq 1}$ such that, for all integers $j\geq 1$, $f_{n_{j}}=\displaystyle \sum_{m\in \mathbb{Z}}\lambda^{n_{j}}_{m}St_{2^{m}}^{(\alpha)} \psi^{n_{j}}_{m}$, $\left(\lambda_{m}^{n_{j}}\right)_{j\geq 1}$ converges in $[0,\infty)$ to $\lambda_{m}$ and  $\left(\psi_{m}^{n_{j}}\chi_{I_{\ell}^{2^{m}}}  \right)_{j\geq 1}$ converges weakly in $L^{p(\cdot)}$ to $\psi_{m,\ell}$.
 	Let us set 
 	\begin{align*}
 		 \psi_{m}=\displaystyle \sum_{\ell \in \mathbb{Z}^{d}}\psi_{m,\ell}, \mbox{ for all } m\in \mathbb{Z}, \mbox{ and } \phi=\displaystyle \sum_{m\in \mathbb{Z}}\lambda_{m}St_{2^{m}}^{(\alpha)}\psi_{m}.
 	\end{align*}
		We have
		\begin{align*}
			 \phi=\lambda_{0}St_{1}^{(\alpha)}\psi_{0} + \displaystyle \sum_{m \geq 1} \left(\lambda_{m}St_{2^{m}}^{(\alpha)}\psi_{m} + \lambda_{-m}St_{2^{-m}}^{(\alpha)}\psi_{-m}\right).
		\end{align*}
		Then,  $\phi$ may be written
		\begin{align*}
			\phi=\displaystyle \sum_{k\geq 0 }c_{k}St_{\rho_{k}}^{(\alpha)}h_{k},
		\end{align*}
		with, for all $m\geq 1$,
		\begin{align*}
			\left( c_{k},\rho_{k},h_{k}\right)= \left\{
			\begin{array}{rl}
				&\left( \lambda_{0},1,\psi_{0}\right) \quad \quad \quad \; \mbox{ if } k=0,\\
				&\left( \lambda_{m},2^{m},\psi_{m}\right) \quad \quad \mbox{ if } k=2m,\\
				&\left( \lambda_{-m},2^{-m},\psi_{-m}\right) \; \mbox{ if } k=2m-1.           
			\end{array}
			\right.
		\end{align*}
		We observe that, for all $k\geq 0$, we have $\rho_{k}>0$ and $\|h_{k}\|_{p(\cdot),q} \leq 1$. Moreover, thanks to the Fatou lemma, we get
	\begin{align*}
		\displaystyle \sum_{k \geq 0}|c_{k}|=\displaystyle \sum_{k\geq 0}
		c_{k}=\displaystyle \sum_{m\in \mathbb{Z}}\lambda_{m}\leq\liminf_{j\rightarrow \infty} \sum_{m\in \mathbb{Z}}\lambda_{m}^{n_{j}}\leq 2^{\frac{d}{q^{\prime}}+1}\,3^{\frac{d}{q}}<\infty.
	\end{align*}
		Hence $\phi$ belongs to $\mathcal{H}(p(\cdot),q,\alpha).$\\
		$\bullet$ Now we shall verify that  $f$ and $\phi$ represent the same element of $L^{p(\cdot)}_{\rm loc}$  and therefore we  will conclude that $f$ belongs to $\mathcal{H}(p(\cdot),q,\alpha)$.
		
		Let us fix  $(\epsilon,j)$ in $(0,\infty)\times \mathbb{N}$ and let $\mathcal{Q}_{0}$ be any dyadic cube in $\mathbb{R}^{d}$. For all $m \in \mathbb{Z}$, we set
		\begin{align*}
			 \varphi_{m}^{n_{j}}=St_{2^{m}}^{(\alpha)}\psi_{m}^{n_{j}} \mbox{ and } \varphi_{m}=St_{2^{m}}^{(\alpha)}\psi_{m}.
		\end{align*}
		Using the decompositions of $f_{n_{j}}$ and $\phi$, the above notations and Notation \ref{not3.14}, we have
		\begin{align}
			\left| \int_{\mathcal{Q}_{0}}f_{n_{j}}(x)dx-\int_{\mathcal{Q}_{0}} \phi(x)dx \right| 
			=& \left| \int_{\mathcal{Q}_{0}} \displaystyle \sum_{m\in \mathbb{Z}}\left( \sum_{\ell \in \mathbb{Z}^{d}} \left[ \lambda_{m}^{n_{j}} \varphi_{m}^{n_{j}}(x)- \lambda_{m} \varphi_{m}(x)\right] \chi_{I_{\ell}^{2^{m}}}(x)\right)dx \right|  \nonumber \\
			\leq & \int_{\mathcal{Q}_{0}}\displaystyle \sum_{(m,\ell)\in \mathcal{D}_{1}(\mathcal{Q}_{0})}  \left|  \lambda_{m}^{n_{j}} \varphi_{m}^{n_{j}}(x)- \lambda_{m} \varphi_{m}(x) \right| \chi_{I_{\ell}^{2^{m}}}(x)dx   \nonumber \\
			&+  \int_{\mathcal{Q}_{0}}\displaystyle \sum_{(m,\ell)\in \mathcal{D}_{2}(\mathcal{Q}_{0})}  \left|  \lambda_{m}^{n_{j}} \varphi_{m}^{n_{j}}(x)- \lambda_{m} \varphi_{m}(x) \right| \chi_{I_{\ell}^{2^{m}}}(x)dx   \nonumber \\  
			&+  \int_{\mathcal{Q}_{0}}\displaystyle \sum_{(m,\ell)\in \mathcal{D}_{3}(\mathcal{Q}_{0})}  \left|  \lambda_{m}^{n_{j}} \varphi_{m}^{n_{j}}(x)- \lambda_{m} \varphi_{m}(x) \right| \chi_{I_{\ell}^{2^{m}}}(x)dx   \nonumber\\  
			=:& \;\Theta_{1}^{j}+\Theta_{2}^{j}+\Theta_{3}^{j}. \label{eq3.10}
		\end{align}
	By virtue of Lemma \ref{lem3.15}, we obtain
		\begin{align}
			\Theta_{1}^{j} &\leq  A_{p}\;2^{\frac{m_{0}d}{q^{\prime}}+m_{1}d\left( \frac{1}{\alpha^{\prime}}-\frac{1}{q^{\prime}}\right)}\left(\displaystyle\sum _{m\leq m_{1}}\lambda_{m}^{n_{j}}+\displaystyle \sum _{m\leq m_{1}}\lambda_{m}\right) \nonumber\\
			&\leq A_{p} \; 2^{\frac{d}{q^{\prime}}+2} \; 3^{\frac{d}{q}}\;  2^{\frac{m_{0}d}{q^{\prime}}+m_{1}d\left( \frac{1}{\alpha^{\prime}}-\frac{1}{q^{\prime}}\right)}, \label{eq3.11}
		\end{align} 
			\begin{align}
				\Theta_{2}^{j}  \leq A_{p} \; 2^{\frac{d}{q^{\prime}}+2} \; 3^{\frac{d}{q}}\;2^{\frac{m_{0}d}{\left(N_{p,2^{m_{0}}}\right)^{\prime}}+m_{2}d\left(\frac{1}{p_{-}}-\frac{1}{\alpha} \right)} \label{eq3.12}
		\end{align} 
		and 
		\begin{align}
			\Theta_{3}^{j} 
			\leq & \, A_{p} \; 2^{\frac{m_{2}d}{\alpha^{\prime}}}\displaystyle \sum_{m_{1}<m<m_{2}}\left|\lambda^{n_{j}}_{m}-\lambda_{m} \right| \nonumber \\
			 & + 2^{\frac{d}{q^{\prime}}+1} \; 3^{\frac{d}{q}} 
			  \displaystyle \sup_{(m,\ell)\in \mathcal{D}_{3}(\mathcal{Q}_{0})}\left| \int_{\mathcal{Q}_{0}}\left[\varphi^{n_{j}}_{m}(x)-\varphi_{m}(x) \right]\chi_{I_{\ell}^{2^{m}}}(x)dx\right|. \label{eq3.13}
		\end{align}
	 Note that the set $\mathcal{D}_{3}(\mathcal{Q}_{0})$  has a finite number of elements. Therefore,
		by \eqref{eq3.11} and \eqref{eq3.12}, we may first choose $m_{1}:=m_{1}(\epsilon)<0$ and $m_{2}:=m_{2}(\epsilon)>0$  such that 
		\begin{align}
			\Theta_{1}^{j}+\Theta_{2}^{j}<\frac{2\epsilon}{3} \label{eq3.14}
		\end{align}
		and after, use \eqref{eq3.13} to choose $j_{\epsilon}$ in $\mathbb{N}$ such that, for all $j\geq j_{\epsilon}$,
		\begin{align}
			\Theta_{3}^{j}<\frac{\epsilon}{3} \label{eq3.15}.
		\end{align}
		A combination of \eqref{eq3.10}, \eqref{eq3.14} and \eqref{eq3.15} leads to
		\begin{align}
			\displaystyle \lim_{j\rightarrow \infty} \left| \int_{\mathcal{Q}_{0}}f_{n_{j}}(x)dx-\int_{\mathcal{Q}_{0}} \phi(x)dx \right|=0. \label{eq3.16}
		\end{align}
	 
		Furthermore, since  $\mathcal{H}(p(\cdot),q,\alpha)$ is continuously embedded in $L^{\alpha}$ (see Proposition \ref{p3.8}), $\left(f_{n_{j}}\right)_{j\geq 1}$ is a norm bounded and non-decreasing sequence of nonnegative elements of $L^{\alpha}$. This implies that $f:=	\displaystyle \sup_{j\geq 1} f_{n_{j}}$ belongs to $L^{\alpha}$ and 
		$\displaystyle\lim_{j\rightarrow \infty}\|f_{n_{j}}-f\|_{\alpha}=0$. Consequently, for any measurable subset $\Omega$ of $\mathbb{R}^{d}$ satisfying $|\Omega|<\infty$, we have
		\begin{align} 
			\displaystyle \lim_{j\rightarrow \infty} \left|\int_{\Omega} f_{n_{j}}(x)dx-\int_{\Omega} f(x)dx\right| =0. \label{eq3.17}
		\end{align}
		
		Since $f$ and $\phi$ are in $L^{\alpha}$, an application of
		the Lebesgue differentiation theorem, combined with \eqref{eq3.16} and \eqref{eq3.17}, show that $f=\phi$ almost everywhere. Thus $f$ belongs to  $\mathcal{H}(p(\cdot),q,\alpha)$.
\end{proof}

Our main result in this subsection is the following.

\begin{theo}\label{th4.19}
		Let $p(\cdot)\in \mathcal{P}(\mathbb{R}^d)$ and $1\leq q <\alpha < p(\cdot)$. Then the space $\mathcal{H}(p(\cdot),q,\alpha)$ satisfies the Fatou property.
\end{theo}

\begin{proof}
	Let $(f_{n})_{n\geq 1}$ be a norm bounded non-decreasing  sequence of nonnegative elements of $\mathcal{H}(p(\cdot),q,\alpha)$ and $f:=\displaystyle \sup_{n\geq 1}f_{n}$.
	From Proposition \ref{p3.16}, $f$ belongs to $\mathcal{H}(p(\cdot),q,\alpha)$. Moreover, since   $\mathcal{H}(p(\cdot),q,\alpha)$ has an absolutely continuous norm (see  Theorem \ref{th3.14}), Point (4) of Proposition \ref{p2.6} implies that  $\displaystyle \lim_{n\rightarrow \infty}\|f_{n}\|_{\mathcal{H}(p(\cdot),q,\alpha)}=\|f\|_{\mathcal{H}(p(\cdot),q,\alpha)}$. This complete the  proof. 
\end{proof}


\subsection{Separability of $\mathcal{H}(p(\cdot),q,\alpha)$} \label{Subsec4.4}

In this subsection we are interested in the separability of the space $\mathcal{H}(p(\cdot),q,\alpha)$. To this end, we need the following result, which established the  separability of the amalgam space $(L^{p(\cdot)},\ell^{q})$.

\begin{prop} \label{p4.20}
	Suppose that $p(\cdot) \in \mathcal{P}(\mathbb{R}^d)$ and $1\leq q < \infty$.	Then the space $(L^{p(\cdot)},\ell^{q})$ is separable.
\end{prop}

\begin{proof}
	We set $\mathfrak{D}= \displaystyle \bigcup_{n\geq 0} \mathfrak{D}_{n}$, where for all nonnegative integers $n$, 
	\begin{align*}
		\mathfrak{D}_{n}=\left\{\sum_{k \in K} c_{k}\chi_{I_{k}^{2^{-n}}} : K \mbox{ is a finite subset of } \mathbb{Z}^{d} \mbox{ and } c_{k} \in \mathbb{Q}+i\mathbb{Q}, \forall \; k\in K\right\}.
	\end{align*}
	We remark that $\mathfrak{D}$ is countable and included in $(L^{p(\cdot)},\ell^{q})$.
	
	Let us consider an element $f$ of $(L^{p(\cdot)},\ell^{q})$ and a real number $\epsilon > 0$. Since the set $\mathcal{C}_{\rm{c}}$ of all continuous and compactly supported functions on $\mathbb{R}^{d}$ is a dense subspace of $\left(L^{p(\cdot)},\ell^{q}\right)$ (see Proposition \ref{P2.1}), there exists $g\in \mathcal{C}_{\rm{c}}$ such that 
	\begin{align} \label{eq4.18}
		\|f-g\|_{p(\cdot),q} < \frac{\epsilon}{2}.
	\end{align}  
	Observe that the set 
	\begin{align*}
		\mathbf{K}=\left\{k\in \mathbb{Z}^{d} : I_{k}^{1} \cap \mbox{supp}(g) \neq\varnothing\right\}
	\end{align*}
	has a finite number of elements that we denote by $N$. Since  the function $g$ is uniformly continuous, there exists a real number $\delta>0$ such that
	\begin{align*}
		\forall \, x,y \in \mathbb{R}^{d}, \quad  \|x-y\|<\delta \Longrightarrow |g(x)-g(y)|<\frac{\epsilon}{4 N},
	\end{align*}
	where $\|\cdot\|$ denotes the usual Euclidean norm on $\mathbb{R}^{d}$.
	
	Let $n$ be a fixed nonnegative integer satisfying $2^{-n} \sqrt{d} < \delta$ and set
	\begin{align*}
		\mathbf{K}(n)=\left\{ \ell \in \mathbb{Z}^{d} : I_{\ell}^{2^{-n}} \cap \mbox{supp}(g) \neq\varnothing\right\}.
	\end{align*}
	Then, for each $\ell \in \mathbf{K}(n)$, we can choose $c_{\ell} \in \mathbb{Q}+i\mathbb{Q}$ such that $ |g(2^{-n}\ell)-c_{\ell}|<\frac{\epsilon}{4 N}$ and therefore
	\begin{align*}
		\forall \, x \in I_{k}^{2^{-n}}, \quad |g(x)-c_{\ell}| \leq |g(x)-g(2^{-n}\ell)|+  |g(2^{-n}\ell)-c_{\ell}|<\frac{\epsilon}{2 N}.
	\end{align*}
	Let us set $h= \displaystyle \sum_{\ell \in \mathbf{K}(n)} c_{\ell}\chi_{I_{\ell}^{2^{-n}}}$. Note that $h$ belongs to $\mathfrak{D}$ and for all $x\in \mathbb{R}^{d}\backslash \left(\displaystyle  \bigcup_{k\in \mathbf{K}}I_{k}^{1}\right)$, we have $g(x)=0=h(x)$. Hence, we get
	\begin{align}
		\|g-h\|_{p(\cdot),q} & =\left[ \displaystyle \sum_{k \in \mathbf{K}} 	\|(g-h)\chi_{I_{k}^{1}} \|_{p(\cdot)}^{q} \right]^{\frac{1}{q}} \leq \left[ \displaystyle \sum_{k \in \mathbf{K}} 	\|(g-h)\chi_{I_{k}^{1}} \|_{p_{+}}^{q} \right]^{\frac{1}{q}} \nonumber \\
		& \leq \left[ \displaystyle \sum_{k \in \mathbf{K}} \left( \displaystyle \sum_{\ell \in \mathbf{K}(n);I_{\ell}^{2^{-n}} \subset I_{k}^{1}} \int_{I_{\ell}^{2^{-n}}} |g(x)-c_{\ell}|^{p_{+}} dx \right)^{\frac{q}{p_{+}}} \right]^{\frac{1}{q}} \nonumber\\
		& < \left[ \displaystyle \sum_{k \in \mathbf{K}} \left( |I_{k}^{1}| \left(\frac{\epsilon}{2 N}\right)^{p_{+}}  \right)^{\frac{q}{p_{+}}} \right]^{\frac{1}{q}} = N^{\frac{1}{q}-1} \times \frac{\epsilon}{2} \leq  \frac{\epsilon}{2} \label{eq4.19}.
	\end{align}
	From \eqref{eq4.18} and \eqref{eq4.19}, we obtain
	\begin{align*}
		\|f-h\|_{p(\cdot),q} < \epsilon.
	\end{align*}
	Hence $\mathfrak{D}$ is dense in $(L^{p(\cdot)},\ell^{q})$ and consequently $(L^{p(\cdot)},\ell^{q})$ is separable.
\end{proof}

Proposition \ref{p4.20} and the density of $(L^{p(\cdot)},\ell^{q})$ in  $\mathcal{H}(p(\cdot),q,\alpha)$ (see Proposition \ref{p2.12}) provide what follows.

\begin{theo} \label{th4.21}
	Let $p(\cdot) \in \mathcal{P}(\mathbb{R}^d)$. Then the space $\mathcal{H}(p(\cdot),q,\alpha)$ is separable.
\end{theo}


\section{Duality results} \label{Sec5}

Under suitable conditions on exponents,  we shall apply the results obtained in the previous sections to determine the K\"{o}the duals and K\"{o}the biduals of $\left(L^{p(\cdot)},\ell^{q}\right)^{\alpha}$ and $\mathcal{H}(p^{\prime}(\cdot),q^{\prime},\alpha^{\prime})$, and a predual space of $\mathcal{H}(p^{\prime}(\cdot),q^{\prime},\alpha^{\prime})$.

In view of Theorem \ref{th3.1}, the Fofana space $\left(L^{p(\cdot)},\ell^{q}\right)^{\alpha}$ satisfies the Fatou property. Therefore Point (3) of Propostion \ref{p2.6} leads to the following corollary.
\begin{coro}\label{cor5.1}
	Let $p(\cdot)$ be a variable exponent on $\mathbb{R}^d$ and  $ p(\cdot) \leq \alpha \leq q \leq \infty$. Then the  space $\left(L^{p(\cdot)},\ell^{q}\right)^{\alpha}$ coincides with its K\"{o}the bidual.  
\end{coro} 

Likewise, Theorem \ref{th4.19} ensures that the  space $\mathcal{H}(p(\cdot),q,\alpha)$ satisfies the Fatou property. Hence, applying Point (3) of Propostion \ref{p2.6}, we obtain the following.
\begin{coro}\label{cor5.2}
	Suppose that $p(\cdot)\in \mathcal{P}(\mathbb{R}^d)$ and $1\leq q <\alpha < p(\cdot)$. Then the space $\mathcal{H}(p(\cdot),q,\alpha)$  coincides with its K\"{o}the bidual.
\end{coro}

Next, we have the following result.
\begin{prop}\label{p5.1}
		Suppose that $p(\cdot) \in \mathcal{P}(\mathbb{R}^d)$ and  $ p(\cdot) \leq \alpha \leq q \leq \infty$. Then $\left(L^{p(\cdot)},\ell^{q}\right)^{\alpha}$ is the  K\"{o}the dual space  of $\mathcal{H}(p^{\prime}(\cdot),q^{\prime},\alpha^{\prime})$.
\end{prop} 

\begin{proof}
	Note that the assumption  $ p(\cdot) \leq \alpha \leq q \leq \infty$ implies that  $1 \leq q^{\prime}\leq \alpha^{\prime} \leq p^{\prime}(\cdot) $. Therefore, since $\mathcal{H}(p^{\prime}(\cdot),q^{\prime},\alpha^{\prime})$ has an absolutely continuous norm (see  Theorem \ref{th3.14}), it follows from Point (4) of Proposition \ref{p2.6} that its  K\"{o}the dual space  $\left[\mathcal{H}(p^{\prime}(\cdot),q^{\prime},\alpha^{\prime})\right]^{\prime}$ coincides with its topological dual space $\left[\mathcal{H}(p^{\prime}(\cdot),q^{\prime},\alpha^{\prime})\right]^{\ast}$. Hence, by Point (1) of Proposition \ref{p2.11}, we get 
	\begin{align*}
\left(L^{p(\cdot)},\ell^{q}\right)^{\alpha}	=\left[\mathcal{H}(p^{\prime}(\cdot),q^{\prime},\alpha^{\prime})\right]^{\ast}=\left[\mathcal{H}(p^{\prime}(\cdot),q^{\prime},\alpha^{\prime})\right]^{\prime}.
	\end{align*}
	The proof is complete.
\end{proof}

\begin{prop}\label{p5.2}
		Suppose that $p(\cdot)\in \mathcal{P}(\mathbb{R}^d)$ and  $ p(\cdot) < \alpha < q \leq \infty$. Then  the  K\"{o}the dual space  of $\left(L^{p(\cdot)},\ell^{q}\right)^{\alpha}$ is $\mathcal{H}(p^{\prime}(\cdot),q^{\prime},\alpha^{\prime})$.
\end{prop} 

\begin{proof}
Note that the assumption  $ p(\cdot) < \alpha < q \leq \infty$ implies that  $1 \leq q^{\prime}< \alpha^{\prime} < p^{\prime}(\cdot) $. Corollary \ref{cor5.1} asserts that $\mathcal{H}(p^{\prime}(\cdot),q^{\prime},\alpha^{\prime})$  coincides with its K\"{o}the bidual space  $\left[\mathcal{H}(p^{\prime}(\cdot),q^{\prime},\alpha^{\prime})\right]^{\prime \prime}$. Therefore, using  Proposition \ref{p5.1}, we obtain 
		\begin{align*}
		\left[ \left(L^{p(\cdot)},\ell^{q}\right)^{\alpha} \right]^{\prime}=\left[\mathcal{H}(p^{\prime}(\cdot),q^{\prime},\alpha)\right]^{\prime \prime}=\mathcal{H}(p^{\prime}(\cdot),q^{\prime},\alpha^{\prime}),
	\end{align*}
	as desired.
\end{proof}

As a consequence of Proposition \ref{p5.2}, we identify a predual space of $\mathcal{H}(p(\cdot),q,\alpha)$.
\begin{prop}
		Suppose that $p(\cdot)\in \mathcal{P}(\mathbb{R}^d)$ and  $ p(\cdot) < \alpha < q \leq \infty$. Then $\left(L^{p(\cdot)},\ell^{q}\right)^{\alpha}_{\rm a}$ is a predual of the  space $\mathcal{H}(p^{\prime}(\cdot),q^{\prime},\alpha^{\prime})$.
\end{prop} 

\begin{proof}
	The carrier of $\left(L^{p(\cdot)},\ell^{q}\right)^{\alpha}_{\rm a}$ is the whole set $\mathbb{R}^{d}$. Indeed, any
	measurable subset $\Omega$ of $\mathbb{R}^{d}$ with $|\Omega|>0$, contains a measurable subset $E$ such that $0<|E|<\infty$ and, $\chi_{E}$ belongs to $L^{\alpha}$ and therefore to $\left(L^{p(\cdot)},\ell^{q}\right)^{\alpha}_{\rm a}$.
	Thus, Point (4) of Proposition \ref{p2.6} and Proposition \ref{p5.2} show that 
\begin{align*}
	\left[\left(L^{p(\cdot)},\ell^{q}\right)^{\alpha}_{\rm a}\right]^{\ast}= \left[\left(L^{p(\cdot)},\ell^{q}\right)^{\alpha}\right]^{\prime}=\mathcal{H}(p^{\prime}(\cdot),q^{\prime},\alpha^{\prime}).
\end{align*}
This ends the proof.
\end{proof}

We end this subsection by the following result.

\begin{prop}
	Let us assume that $p(\cdot) \in \mathcal{P}(\mathbb{R}^d)$ and $ p(\cdot) <  \alpha < q \leq \infty$. Then both $\left(L^{p(\cdot)},\ell^{q}\right)^{\alpha}$ and $\mathcal{H}(p'(\cdot),q',\alpha')$ are not reflexive.
\end{prop}

\begin{proof}
	Under the condition $ p(\cdot) <  \alpha < q \leq \infty$, Propositions \ref{p5.1} and \ref{p5.2} show that  each of both the spaces $\left(L^{p(\cdot)},\ell^{q}\right)^{\alpha}$ and $\mathcal{H}(p'(\cdot),q',\alpha')$ is the K\"{o}the dual space  of the other. These two spaces satisfy the Fatou property, but $\left(L^{p(\cdot)},\ell^{q}\right)^{\alpha}$ fails to have absolutely continuous norm (see Remark \ref{R3.4}). Thus, it follows from Point 4(e) of Proposition \ref{p2.6} that both $\left(L^{p(\cdot)},\ell^{q}\right)^{\alpha}$ and $\mathcal{H}(p'(\cdot),q',\alpha')$ are not reflexive.
\end{proof}

\section{Conclusion} \label{Sec6}
 Let $1\leq p(\cdot) \leq \alpha \leq q\leq \infty$. In this paper, we have systematically investigated the normed K\"{o}the structure of  Fofana spaces with variable exponent $(L^{p(\cdot)}, \ell^{q})^{\alpha}$ as well as their predual spaces $\mathcal{H}(p^{\prime}(\cdot),q^{\prime},\alpha^{\prime})$. Under suitable conditions on exponents, we focused on the study of two fundamental properties. On the one hand, we have established that both $(L^{p(\cdot)}, \ell^{q})^{\alpha}$ and $\mathcal{H}(p^{\prime}(\cdot),q^{\prime},\alpha^{\prime})$ satisfy the Fatou property. On the other hand, we have proved that the space $(L^{p(\cdot)}, \ell^{q})^{\alpha}$ does not have an absolutely continuous norm, whereas $\mathcal{H}(p^{\prime}(\cdot),q^{\prime},\alpha^{\prime})$  has an absolutely continuous norm. Moreover, we have showed that the space $\mathcal{H}(p^{\prime}(\cdot),q^{\prime},\alpha^{\prime})$ is separable. As consequences, we obtained some duality results between these spaces, including characterizations of their K\"{o}the duals, their K\"{o}the biduals and those of a predual of the space $\mathcal{H}(p^{\prime}(\cdot),q^{\prime},\alpha^{\prime})$. Finally, we have proved that neither  $(L^{p(\cdot)}, \ell^{q})^{\alpha}$ nor $\mathcal{H}(p^{\prime}(\cdot),q^{\prime},\alpha^{\prime})$ is reflexive. 
 
 Our present study contributes to a deeper understanding of the normed K\"{o}the space structure of variable-exponent Fofana spaces and their preduals. In particular, the results established herein provide a solid analytical foundation for extending several classical theorems in harmonic analysis to this broader framework.
 For instance, as a direction for future research, we plan to investigate the interpolation spaces associated with variable-exponent Fofana spaces and their preduals. Such an investigation may provide further insight into their functional-analytic structure and contribute to the development of interpolation theory in this setting.



\end{document}